\documentclass{amsart}
\usepackage{tkz-euclide}
\usepackage{memhfixc}
\usepackage{latexsym}
\usepackage{amssymb}
\usepackage{tikz}
\usepackage{amsthm}
\usepackage{amsmath}
\usepackage{enumerate}
\usepackage{esint}
\usepackage{tkz-euclide}
\usepackage{tikz}
\usetikzlibrary{decorations.fractals}
\usetikzlibrary{decorations.footprints}
\usetikzlibrary{through, intersections,decorations.text}
\usetikzlibrary{%
  arrows,
  calc
}

\newtheorem{theorem}{Theorem}[section]
\newtheorem{lemma}[theorem]{Lemma}

\theoremstyle{definition}

\newtheorem{corollary}[theorem]{Corollary}
\theoremstyle{remark}
\newtheorem{remark}[theorem]{Remark}

\numberwithin{equation}{section}

\newcommand\rd{\mathrm{d}}
\newcommand\ri{\mathrm{i}}
\newcommand\re{\mathrm{e}}
\font\ursymbol=psyr at 10pt 
\def\urpartial{\mbox{\ursymbol\char"B6}}

\DeclareMathOperator*{\osc}{osc}
\DeclareMathOperator*{\esssup}{ess\,sup}
\DeclareMathOperator*{\essinf}{ess\,inf}

\def\XXint#1#2#3{{\setbox0=\hbox{$#1{#2#3}{\int}$ }
\vcenter{\hbox{$#2#3$ }}\kern-.55\wd0}}

\newcommand{\hd}{\mbox{H-dim}\;}

\newcommand{\ce}{ \mathbb{C}}
\newcommand{\beq}{\begin{equation}}
\newcommand{\bea}[1]{\begin{array}{#1} }
\newcommand{\eeq}{ \end{equation}}
\newcommand{\ea}{ \end{array}}

\newcommand{\ran}{\rangle}
\newcommand{\lan}{\langle}

\newcommand{\la}{\lambda}

\newcommand{\ar}{\partial}
\newcommand{\si}{\sigma}

\newcommand{\Om}{\Omega}

\begin{document}

\title{On the Dimension of a Certain Measure In the Plane}

\author[Murat Akman]{Murat Akman}
\address{Department of Mathematics \\ University of  Kentucky  \\ Lexington, Kentucky, 40506}
\email{makman@ms.uky.edu}



\date{}




\dedicatory{Dedicated to John L. Lewis on the occasion of his 70th birthday}


\begin{abstract}
In this paper we study the Hausdorff dimension of a measure $\mu$ related to a positive weak solution, $u$, 
of a certain partial differential equation in $\Omega\cap N$ where $\Omega\subset \mathbb{C}$ is a bounded simply 
connected domain and $N$ is a neighborhood of $\partial\Omega$. $u$ has continuous boundary value $0$ on $\partial\Omega$ and is a weak solution to 
\[
 \sum\limits_{i,j=1}^{2}\frac{\partial}{\partial x_{i}}(f_{\eta_{i}\eta_{j}}(\nabla u(z))\, u_{x_{j}}(z))=0\, \ \mbox{in}\, \ \Omega\cap N.
\]
Also $f(\eta)$, $\eta\in\mathbb{C}$ is homogeneous of degree $p$ and $\nabla f$ is $\delta-$monotone on $\mathbb{C}$ for some $\delta>0$.
Put $u\equiv 0$ in $N\setminus \Omega$. Then $\mu$ is the unique positive finite Borel measure with support on $\partial\Omega$ satisfying 
\[
\int\limits_{\mathbb{C}}\langle \nabla f(\nabla u(z)), \nabla\phi(z)\rangle dA=-\int\limits_{\partial\Omega} \phi(z) d\mu
\]
for every $\phi\in C^{\infty}_{0}(N)$.

Our work generalizes work  of  Lewis and coauthors when the above PDE is the $p$ Laplacian  
(i.e, $f(\eta)=|\eta|^p$) and also for $p=2$, the well known theorem of Makarov regarding the Hausdorff 
dimension of harmonic measure relative to a point in $\Omega$.
\end{abstract}
\maketitle
\section{Introduction}
\label{introduction}

Let $\Omega'$ denote a  bounded region in the complex plane $\ce$. 
Given $p$, $1<p<\infty$, let $z=x_1+\ri x_2$ denote points in $\mathbb{C}$ and let $W^{1,p}(\Omega')$ denote equivalence classes of
functions $h:\mathbb{C}\to\mathbb{R}$ with  distributional gradient  $\nabla h =h_{x_1}+\ri h_{x_2}$ and  Sobolev norm


\begin{align}
\|h\|_{W^{1,p}(\Omega')}=\left(\int\limits_{\Omega'}(|h|^{p}+|\nabla h|^{p})\rd\nu\right)^{\frac{1}{p}}<\infty 
\end{align} 
where $\rd\nu$ denotes two dimensional Lebesgue measure. The space $W^{1,p}_{\mbox{\tiny{loc}}}(\Omega')$ 
is defined in the obvious manner; $h\in W^{1,p}_{\mbox{\tiny{loc}}}(\Omega')$ if and only if $h\in W^{1,p}(U)$ for 
every open $U\Subset \Omega'$, i.e compactly contained in $\Omega'$.

Let $ C_0^\infty(\Omega')$ denote infinitely differentiable functions with compact support in $\Omega'$ and 
let $ W^{1,p}_0(\Omega')$ denote the closure of $C_0^\infty(\Omega')$ in the norm of  $W^{1,p}(\Omega')$. 
Let $\langle\cdot,\cdot\rangle$ denote the standard
inner product on $\mathbb{C}$.

Fix $p$, $1<p<\infty$ and 
let  $f:\mathbb{C}\setminus\{0\}\to(0,\infty)$ be homogeneous of degree $p$ on $\mathbb{C}\setminus\{0\}$. That is,    


\begin{align}
\label{homogeneousofdegreep}
f(\eta)=|\eta|^{p}f(\frac{\eta}{|\eta|})>0\, \, \mathrm{when}\, \, \eta\in\mathbb{C}\setminus\{0\}.
\end{align}


We also assume that $\nabla f$ is $\delta-$monotone on $\mathbb{C}$ for some $0<\delta\leq 1$. 
By definition, this means that $f\in W^{1,1}(B(0,R))$ for each $R>0$
and for almost every $\eta, \eta'\in\mathbb{C}$ (with respect to two dimensional Lebesgue measure) 
\begin{align}
\label{deltamonotone}
 \begin{split}
  \langle \nabla f(\eta)-\nabla f(\eta'), \eta-\eta'\rangle\geq \delta|\nabla f(\eta)-\nabla f(\eta')||\eta-\eta'|.
 \end{split}
\end{align}

Next, given $h\in W^{1,p}(\Omega')$ let  $\mathfrak{A}=\{h+\phi:\,\phi\in W^{1,p}_{0}(\Omega')\}$.  
From (\ref{star}) in section \ref{lemmas} and \cite[Chapter 5]{HKM} it follows that


\begin{align}
\label{weaksoln}
\inf_{w\in\mathfrak{A}}\,\int\limits_{\Omega'}f(\nabla w)\rd\nu=\int\limits_{\Omega'}f(\nabla u')\rd\nu\, \ \mathrm{for\ some}\, \ u'\in \mathfrak{A}.
\end{align}


Also $u'$ is a weak solution at $z\in\Omega'$ to the Euler-Lagrange equation,


\begin{align}
\label{flaplace} 
\begin{split}
0&=\nabla\cdot(\nabla f(\nabla u'(z)))=\sum\limits_{k=1}^{2}\frac{\ar}{\partial x_{k}}\left(\frac{\ar f}{\partial\eta_k}(\nabla u'(z))\right) \\
 &=\sum\limits_{k,j=1}^{2}f_{\eta_{k}\eta_{j}}(\nabla u'(z))\, u'_{x_{k} x_{j}}(z)   
\end{split}
\end{align}


That is, $u'\in W^{1,p}(\Omega')$ and 


\begin{align}
\label{uisweaksoln}
\int\limits_{\Omega'}\lan\nabla f(\nabla u'(z)), \nabla\phi(z)\ran \rd\nu=0\,\  \mathrm{whenever}\, \ \phi\in W^{1,p}_0(\Omega').   
\end{align}

Next, suppose $\Omega\subset\ce$ is a bounded simply connected domain, $N$ is a neighborhood of $\partial\Omega$, and $u>0$ is a weak solution to the 
Euler Lagrange equation in (\ref{flaplace}) with $\Omega'=\Omega\cap N$, $u'=u$. Also assume that $u=0$ on $\partial\Omega$ in the $W^{1,p}(\Omega\cap N)$ sense. More specifically, let $u\equiv 0$ on $N\setminus\Omega$. Then $u\zeta\in W^{1,p}_{0}(\Omega)$ whenever $\zeta\in C^{\infty}_{0}(\Omega)$. Under this scenario 
it follows from \cite[Chapter 21]{HKM} that there exists 
a unique finite positive Borel measure $\mu$ with support on $\partial\Omega$ satisfying


\begin{align}
\label{ast}
\begin{split}
\int\limits_{\ce}\lan\nabla f(\nabla u(z)), \nabla\phi\ran \rd\nu=-\int\limits_{\partial\Omega}\phi \rd \mu
\end{split}
\end{align}
whenever $\phi\in C_{0}^{\infty}(N)$.
\begin{remark}
 We remark from (\ref{ast}) that if $\partial\Omega$ and $f$ are smooth enough then 
\[
 \rd\mu=\frac{f(\nabla u)}{|\nabla u|}\rd H^{1}|_{\partial\Omega}.
\]
\end{remark}
We are now ready to introduce the notions of Hausdorff measure and Hausdorff dimension of $\mu$ associated with a 
weak solution $u$ to (\ref{flaplace}) in $\Omega\cap N$.

Let $\lambda>0$ be defined on $(0, r_{0})$ with $\lim\limits_{r\to 0}\lambda(r)=0$ for some fixed $r_{0}$. 
We define the $H^{\lambda}$ measure of a set $E\subset\mathbb{C}$ as follows; 

For fixed $0<\delta<r_{0}$, let $\{B(z_{i}, r_{i})\}$ be a cover of $E$ with $0<r_{i}<\delta$, $i=1,2,\ldots$, and set
\[
 \phi_{\delta}^{\lambda}(E)=\inf\sum\limits_{i}\lambda(r_{i}).
\]
where the infimum is taken over all possible covers of $E$. 

Then the Hausdorff $H^{\lambda}$ measure of $E$ is
\[
 H^{\lambda}(E)=\lim\limits_{\delta\to 0} \phi_{\delta}^{\lambda}(E).
\]
When $\lambda(r)=r^{\alpha}$ we write $H^{\alpha}$ for $H^{\lambda}$. 
Next we define the Hausdorff dimension of the measure $\mu$ obtained in (\ref{ast}) as 
\[
 \hd{\mu}=\inf\{\alpha:\, \ \exists\, \mathrm{Borel\ set}\, \ E\subset\partial\Omega\, \ \mathrm{with}\, \ H^{\alpha}(E)=0\, \ \mathrm{and}\, \ \mu(E)=\mu(\partial\Omega)\}.
\]

To give a little history, if $\mu=\omega$ is harmonic measure with respect to a point $z_{0}\in\Omega$, the case when $f(\nabla u)=|\nabla u|^{2}$ and
$u$ is a solution to Laplace's equation in $\Omega\setminus\{z_{0}\}$, then Carleson showed in \cite{C} that
\begin{theorem}
\label{carleson}
$\hd{\omega}=1$ when $\partial\Omega$ is a snowflake and $\hd{\omega}\leq 1$ when $\Omega$ is any self similar cantor set. 
\end{theorem}

In \cite{M}, Makarov proved that
\begin{theorem} 
Let $\Omega$ be a simply connected and $\mu=\omega$ in (\ref{ast}) be harmonic measure with respect to
a point in $\Omega$, and let
\[
 \lambda(r)=r\ \mathrm{exp}\{A\sqrt{\log\frac{1}{r} \log\log\log\frac{1}{r}}\},\,\ \ 0<r<10^{-6}.
\]
Then there exists an absolute constant $A>0$ such that harmonic measure $\omega$ is absolutely 
continuous with respect to Hausdorff $H^{\lambda}$ measure.
\end{theorem}

In \cite{JW}, Jones and Wolff proved that 
\begin{theorem}
\label{jonesandwolfftheorem}
$\hd{\omega}\leq 1$ for an arbitrary domain $\Omega$ in the plane when $\omega$ exists.  
\end{theorem}
Later Wolff in \cite{W} extended Theorem \ref{jonesandwolfftheorem} by proving
\begin{theorem}
 Harmonic measure $\omega$ is concentrated on a set of $\sigma-$finite $H^{1}$ measure whenever $\Omega$ is an arbitrary planar domain for which $\omega$ exists.
\end{theorem}

In \cite{BL}, Bennewitz and Lewis obtained the following result
for $\mu$ defined as in (\ref{ast}) for fixed $p$, $1<p<\infty$, relative to $f(\nabla u)=|\nabla u|^{p}$.
In this case the corresponding pde (\ref{flaplace}) becomes
\begin{align}
\label{plaplace}
\nabla\cdot(|\nabla u|^{p-2}\nabla u)=0,
\end{align}
which is called the $p-$Laplace equation. Moreover a weak solution of (\ref{plaplace}) is called a $p-$harmonic function.

\begin{theorem} 
Let  $\Omega\subset\mathbb{C}$ be a domain bounded by a quasi circle and let $N$ be a neighborhood of $\partial\Omega$.
Fix $p\not = 2$, $1<p<\infty$, and suppose $u$
is  $p$-harmonic in $\Omega\cap N$ with boundary value $0$ in the $W^{1,p}(\Omega\cap N)$ Sobolev sense. If $\mu$ is the measure corresponding to
$u$ as in (\ref{ast}) relative to  $f(\nabla u)=|\nabla u|^{p}$, then $\hd{\mu}\leq 1$ for  $2<p<\infty$ while $\hd{\mu}\geq 1$ for $1<p<2$.
Moreover, if $\partial\Omega$ is the von Koch snowflake then strict inequality holds for $\hd{\mu}$.
\end{theorem}


In \cite{LNP}, Lewis, Nystr\"om, and Poggi-Corradini proved that


\begin{theorem} Let  $ \Om \subset \mathbb{C} $ be a bounded simply connected domain and $ N $ a neighborhood of  $ \ar \Om. $
Fix $ p\not = 2, 1 < p < \infty, $ and let $ u $
be $p$ harmonic in $ \Om \cap N $ with boundary value 0 on $\partial\Omega$ in the $ W^{1,p} ( \Om \cap N ) $ Sobolev sense.  Let  $\mu$ be the measure corresponding to
$u$ as in (\ref{ast}), relative to $f(\nabla u)=|\nabla u|^{p}$ and put \[
\la(r)=r\,\exp   \left[ A\sqrt{\log\frac{1}{r}\,\log\log\frac{1}{r}} \right]\,\ \mathrm{for}\, \ 0<r<10^{-6} .
\]
\begin{itemize}
\item[a)] If $p>2$, there exists $A=A(p)\leq -1$ such that $\mu$ is concentrated on a set of $\si-$finite $H^{\la}$ Hausdorff measure.
\item[b)] If  $1<p<2$, there exists $A=A(p)\geq 1$, such that $\mu$ is absolutely continuous with respect to $H^{\la}$ Hausdorff measure.
\end{itemize}
\end{theorem}

In the recent paper \cite{L}, Lewis proved that


\begin{theorem} Let $\Omega\subset\mathbb{C}$ be a bounded simply connected domain and $N$ be a neighborhood of  $\partial\Omega$.
Fix $p\not = 2$, $1<p<\infty$, and let $u$
be $p-$harmonic in $\Omega\cap N$ with boundary value $0$ on $\partial\Omega$ in the $W^{1,p}(\Omega\cap N)$ Sobolev sense. Let $\mu$ be the measure corresponding to
$u$ as in (\ref{ast}), relative to $f(\nabla u)=|\nabla u|^{p}$ and put
\[
\tilde{\lambda}(r)=r\,\exp\left[ A\sqrt{\log\frac{1}{r}\,\log\log\log\frac{1}{r}} \right]\,\ \mathrm{for}\, \ 0<r<10^{-6} .
\]
\begin{itemize}
\item[a)] If $p>2$, then $\mu$ is concentrated on a set of $\si-$finite $H^{1}$  measure.
\item[b)] If  $1<p<2$, there exists $A=A(p)\geq 1$, such that $\mu$ is absolutely continuous with respect to $H^{\tilde{\lambda}}$  measure.
Moreover $A(p)$ is bounded on $(3/2,2)$.
\end{itemize}
\end{theorem}

This theorem is the complete extension of Makarov's theorem  to  the $p$-harmonic setting.

In this paper we obtain that, 


\begin{theorem}
\label{maintheorem}
Let  $\Omega\subset\mathbb{C}$ be a bounded simply connected domain and let $N$ be a neighborhood of  $\partial\Omega$.
Fix $p$, $1<p<\infty$, let $f$ be homogeneous of degree $p$ and let $\nabla f$ be $\delta$ monotone for some $0<\delta\leq 1$. Let  $\hat{u}>0$ be a weak solution to
(\ref{flaplace}) in $\Omega\cap N$ with boundary value $0$ on $\partial\Omega$ in the $W^{1,p}(\Omega\cap N)$ Sobolev sense. Let $\hat{\mu}$ be the measure corresponding to
$\hat{u}$ as in (\ref{ast}) and put
\[
\lambda(r)=r\,\exp   \left[ A\sqrt{\log\frac{1}{r}\,\log\log\frac{1}{r}} \right]\,\ \mathrm{for}\, \ 0<r<10^{-6}.
\]
\begin{itemize}
\item[a)] If $p\geq 2$, there exists $A=A(p)\leq -1$ such that $\hat{\mu}$ is concentrated on a set of $\si-$finite $H^{\la}$ Hausdorff measure.
\item[b)] If  $1<p\leq 2$, there exists $A=A(p)\geq 1$, such that $\hat{\mu}$ is absolutely continuous with respect to $H^{\la}$ Hausdorff measure.
\end{itemize}
\end{theorem}
Note that Theorem \ref{maintheorem} and the definition of $\hd{\hat{\mu}}$ imply the following corollary.

\begin{corollary}
Given $p$, $1<p<\infty$, let $\hat{u},\hat{\mu}$ be as in Theorem \ref{maintheorem}, and suppose $\Omega$ is a simply connected domain. Then
$\hd{\hat{\mu}}\leq 1$ for $2\leq p<\infty$, while $\hd{\hat{\mu}}\geq 1$ for $1<p\leq 2$. 
\end{corollary}

This paper is organized as follows.
In section \ref{lemmas} we obtain some regularity results for $f$ and $u$. Indeed, in subsection \ref{basicregresult} we first introduce some notation which we will 
use throughout this paper and we mention some regularity properties of  $f$ satisfying  (\ref{homogeneousofdegreep}) and (\ref{deltamonotone})
suitable for use in elliptic regularity theory.

In subsection \ref{advreg} we study a variational problem and indicate some properties of  
weak solutions to the  corresponding Euler Lagrange equation: maximum principle, Harnack inequality, interior H\"{o}lder continuity of a solution,
and H\"{o}lder continuity near the boundary of $\Omega$.
After that we study the behavior of $\hat{u}$ near $\partial\Omega$ and the relationship between $\hat{u}$ and $\hat{\mu}$ as in (\ref{ast}).
Using this relationship  we see that $\hd{\hat{\mu}}$ is independent of the corresponding $\hat{u}$. 

In subsection \ref{madvreg}, we use elliptic and quasiregularity theory to derive more advanced regularity properties of $\hat{u}$: 
quasiregulariy of  $\hat{u}_{z}$,  H\"{o}lder continuity of $\nabla \hat{u}$, and  $\nabla \hat{u}$ locally in $W^{1,2}$, 
so $\hat{u}$ is  almost everywhere a pointwise solution to (\ref{flaplace}). We also show for a certain $u$ 
that $\nabla u\neq 0 $ near $\partial\Omega$. Next we outline a proof in \cite{LNP} which shows for a certain $u$ as in 
Theorem \ref{maintheorem} that
 $\nabla u$ satisfies the so called fundamental inequality. Using this inequality and previous results we first obtain that
$u$ and $\nabla u$ are weak solutions to a certain pde and then that $\log f(\nabla u) $ is a weak sub, super or solution to this pde, depending on whether $p>2, <2$, or $=2$.

In section \ref{proofofmain} we prove Theorem \ref{maintheorem}.

In general we follow the game plan of Lewis and coauthors who in turn were influenced by the work of Makarov. 
However the equation we consider is more complicated and has less regularity than the $p$ Laplacian. Thus we had to overcome numerous procedural difficulties 
not encountered in \cite{LNP}.

\section{Some Lemmas}
\label{lemmas}
Throughout this paper various positive constants are denoted by $c$ and they may differ even on the same line. 
The dependence on parameters is expressed, for example, by $c=c(p,f)\geq 1$. Also  
$g\approx h$ means that there is a constant $c$ such that
\[
\tfrac{1}{c} h \leq g \leq c\, h.
\]
Let $B(z,r)$ denote the disk in $\mathbb{R}^{2}$ or $\mathbb{C}$
with center $z$ and radius $r$ and let $\nu$ be two dimensional Lebesgue measure. 

Let $\eta=\begin{bmatrix}
                \eta_{1} \\
		 \eta_{2}
         \end{bmatrix}
$
be a 2 x 1 column matrix and let $\eta^{\mbox{\tiny{T}}}=\begin{bmatrix} \eta_{1} & \eta_{2}\end{bmatrix}$ denote the 
transpose of $\eta$. 
We specifically denote the unit disk, $B(0,1)$, by $\mathbb{D}$.
$\Omega$ will always denote an open set and often $\Omega$ is a simply connected domain. 
That is $\Omega$ is an open connected domain whose complement is connected. 
\subsection[Basic Regularity Results for f]{Basic Regularity Results for $\boldsymbol{f}$}
\label{basicregresult}

In this subsection we state some regularity result for $f$. Let $f$ be as in (\ref{homogeneousofdegreep}), (\ref{deltamonotone}). Then $\nabla f$ has a 
representative in $L^{1}(\mathbb{C})$ (also denoted by $\nabla f$) that is $\delta-$monotone on $\mathbb{C}$. From homogeneity of $f$ and 
Kovalev's theorem in \cite{K} we see that $\nabla f$
is in fact a $K-$quasiconformal mapping in $\mathbb{C}$ where

\[
 K=\frac{1+\sqrt{1-\delta^{2}}}{1-\sqrt{1-\delta^{2}}}.
\]
So the eigenvalues of the Hessian matrix of $\nabla f$ either both exist and are zero or have ratios bounded above by $K$ and below by $1/K$.

As $f$ is homogeneous of degree $p$, i.e $f(\eta)=|\eta|^{p}f(\eta/|\eta|)$ when $\eta\in\mathbb{C}\setminus\{0\}$, if we introduce polar 
coordinates; $r=|\eta|$, $\tan(\theta)=\eta_{2}/\eta_{1}$, then 
\[
f(r,\theta)=r^{p}f(\cos(\theta), \sin(\theta)).
\]
Hence first and second derivatives of $f$ along rays through the origin are 
\begin{align}
\label{frandfrr}
f_{r}=pr^{p-1}f(\cos(\theta), \sin(\theta))\, \, \, \mathrm{and}\, \, \, f_{rr}=p(p-1)r^{p-2}f(\cos(\theta), \sin(\theta)).
\end{align}
$K-$quasiregularity of $\nabla f$ implies that $f$ is continuous in $\mathbb{C}$. Since $f>0$ it 
follows that $f(\cos(\theta), \sin(\theta))$ is bounded above and below by constants $1\leq M$ and $1/M$ respectively. 
We conclude from this fact and (\ref{frandfrr}) that  
\begin{align}
\label{frrbddM}
\tfrac{1}{M}p(p-1)r^{p-2}\leq f_{rr}\leq M\, p(p-1)r^{p-2}.
\end{align}
From (\ref{frrbddM}) and $K-$quasiregularity of $\nabla f$  it follows for a.e. $\eta\in\mathbb{C}$ and all $\xi$ with $|\xi|=1$ that
\begin{align}
\label{frrbddKM}
\tfrac{1}{MK}p(p-1)|\eta|^{p-2}\leq f_{rr}\approx f_{\xi\xi}(\eta)=\eta^{\mbox{\tiny{T}}}\, D^{2}f\, \eta\leq MK\, p(p-1)|\eta|^{p-2}
\end{align}
where $D^{2}f=(f_{\eta_{i}\eta_{j}})$. It follows from homogeneity of $f$ and (\ref{frrbddKM}) for some $M'\geq 1$ that
\begin{align}
\label{frrM'}
\tfrac{1}{M'}|\eta|^{p}\leq \min\{f(\eta), |\eta||\nabla f(\eta)|\}\leq\max\{f(\eta), |\eta||\nabla f(\eta)|\}\leq M'|\eta|^{p}.
\end{align}
Using (\ref{frrbddM}) we also see from basic calculations that for $\eta, \eta'\in\mathbb{C}$,
\begin{align}
\label{star}
 \begin{split}
  \tfrac{1}{c} (|\eta|+|\eta'|)^{p-2} |\eta-\eta'|^{2}\leq  \langle\nabla f(\eta)-\nabla f(\eta'), \eta-\eta'\rangle \leq c(|\eta|+|\eta'|)^{p-2} |\eta-\eta'|^{2}.
 \end{split}
\end{align}
Let $\theta(z)$ be the standard mollifier, i.e;
\begin{align*}
 \begin{split}
\theta(z)=\left\{\begin{array}{ll}
           c\,\,  \mathrm{exp}(\frac{1}{|z|^{2}-1}) & \mathrm{if}\, \ |z|<1 \\
	               0 & \mathrm{if}\, \ |z|\geq 1
          \end{array}
\right.
 \end{split}
\end{align*}
Let $f_{\varepsilon}=f\ast \theta_{\varepsilon}$ where
\begin{align}
\label{convprop}
 f_{\varepsilon}(z)=\int\limits_{\mathbb{C}}\theta_{\varepsilon}(z-w)f(w)\rd w = \int\limits_{B(0,\varepsilon)}\theta_{\varepsilon}(w)f(z-w)\rd w
\end{align}
for $z\in\mathbb{C}$. For later use we note that (\ref{star}) and the definition of $f_{\varepsilon}$ easily imply
\begin{align}
\label{epsilonstar}
 \begin{split}
  \tfrac{1}{c} (|\eta|+|\eta'|+\varepsilon)^{p-2} |\eta-\eta'|^{2}\leq  
\langle\nabla f_{\varepsilon}(\eta)-\nabla f_{\varepsilon}(\eta'), &\eta-\eta'\rangle \\ \leq 
&c\, (|\eta|+|\eta'|+\varepsilon)^{p-2} |\eta-\eta'|^{2}.
 \end{split}
\end{align}

Finally, we state for further use a lemma which is a direct consequence of (\ref{frrbddKM}) and (\ref{frrM'}) for $u\in W^{1,1}(\Omega)$.
\begin{lemma}
\label{fishomdp}
For some constants $c, c', c''\geq 1$ depending only on $f$, we have for a.e $z\in \Omega$,
\begin{eqnarray*}
&&\frac{1}{c}|\nabla u|^{p}\leq f(\nabla u) \leq c |\nabla  u|^{p},\\
&&\label{f1} \frac{1}{c'}|\nabla u|^{p-1}\leq|\nabla f(\nabla  u)| \leq c' |\nabla  u|^{p-1},\\
&&\label{f2} \frac{1}{c''}|\nabla u|^{p-2}\leq\|D^{2} f(\nabla  u)\| \leq c'' |\nabla  u|^{p-2},
\end{eqnarray*}
where $\|D^{2}f(\nabla u)\|$ denotes the absolute value of an arbitrary second derivative of $f$ evaluated at $\nabla u(z)$.
\end{lemma}
\subsection{Interior and boundary estimates for $\hat{u}$}
\label{advreg}
We refer to \cite{BL} for references to the proofs of Lemmas \ref{caccioppolitypeinequality} - \ref{uisholder}.

Let $w\in\partial\Omega$ and $0<r<\, \mbox{diam}\, \, \Omega$. Moreover, let $f$ be as in Theorem \ref{maintheorem}. We also put $f(0)=0$. 
In this subsection we begin by stating some interior and boundary estimates for $\hat{u}$ a positive weak solution to (\ref{flaplace}) in $B(w,4r)\cap\Omega$
with $\hat{u}=0$ on $B(w,4r)\cap\partial\Omega$ in the Sobolev sense.
\begin{lemma}
\label{caccioppolitypeinequality}
For fixed $p$, $1<p<\infty$, let $\hat{u}, f, \Omega, w, r$ be defined as above. Then
\begin{align}
 \label{21}
\tfrac{1}{c}r^{p-2}\int\limits_{B(w, \frac{r}{2})} f(\nabla \hat{u}) \rd\nu  \leq \esssup\limits_{B(w,r)} \hat{u}^{p} \leq c\, \frac{1}{r^{2}}\int\limits_{B(w,r)}\hat{u}^{p}\rd\nu.
\end{align}
\end{lemma}
\begin{lemma}[Harnack's Inequality]
\label{harnackinequality}
Let $\hat{u},\Omega, r, w$ be as in Lemma \ref{caccioppolitypeinequality}. Then there is a constant $c=c(p,f)$ such that
\begin{align}
\esssup\limits_{B(\tilde{w},s)}\hat{u} \leq c\essinf\limits_{B(\tilde{w},s)} \hat{u}.
\end{align}
whenever $B(\tilde{w}, 2s)\subset B(w,4r)\cap\Omega$.
\end{lemma}
Next we state local H\"older continuity of $\hat{u}$.
\begin{lemma}
\label{uisholder}
Let $\hat{u},\Omega, w, r$ be as in Lemma \ref{caccioppolitypeinequality}. 
Let $0<s_{0}<\infty$ and suppose that $B(w_{0},s_{0})\subset B(w,4r)\cap\Omega$. 
Then for $0<s<s_{0}$ there is a constant
$0<\alpha=\alpha(p,f)\leq 1$ such that
\[
 |\hat{u}(\tilde{w})-\hat{u}(\hat{w})|\leq c \left(\frac{|\tilde{w}-\hat{w}|}{s}\right)^{\alpha} \esssup\limits_{B(w_{0}, s_{0})} \hat{u}.
\]
\end{lemma}
Next we indicate H\"older continuity of $\hat{u}$ near $B(w,4r)\cap\partial\Omega$.
\begin{lemma}[Behavior of $\hat{u}$ near the boundary]
\label{boundaryharnackinequalty}

Let $\hat{u}, \Omega, w, r$ be as in Lemma \ref{caccioppolitypeinequality}. Then there is $\alpha'=\alpha'(p,f)>0$ 
such that $\hat{u}$ has a H\"older continuous representative in $B(w,r)$  and if $\tilde{w},\hat{w}\in B(w,r)$ then
\begin{align}
\label{boundaryharnackinequaltyofu}
\begin{split}
  |\hat{u}(\tilde{w})-\hat{u}(\hat{w})|\leq c \left(\frac{|\tilde{w}-\hat{w}|}{r}\right)^{\alpha'}\esssup\limits_{B(w,2r)} \hat{u}.
 \end{split}
\end{align}
\end{lemma}
\begin{proof}
The proof for $p>2$ follows from Lemma \ref{caccioppolitypeinequality} and Morrey's Theorem. 
For $1<p\leq 2$ we note that there is a continuum $\subset \overline{B(w,t)}\setminus \Omega$ connecting $w$ to $\partial B(w,t)$ as follows from 
simply connectivity of $\Omega$. We also note that this continuum is uniformly fat in the sense of $p-$capacity (see \cite{L88} 
for the definition of a uniformly fat set). That is, the $p-$capacity of this continuum is $\geq c^{-1}$ times the $p-$capacity of $B(w,r)$.
Using this fact in the Wiener integral in \cite[Theorem 6.18]{HKM} we obtain for $0<\rho\leq r/2$
\begin{align}
\label{maxinballismaxin2b}
 \begin{split}
\osc\limits_{B(w,\rho)\cap\Omega} \hat{u} \leq c\left(\frac{\rho}{r}\right)^{\alpha'} \esssup\limits_{B(w,r)}\hat{u}
\end{split}
 \end{align}
for some $c=c(p,f,\Omega)>0$. From (\ref{maxinballismaxin2b}) 
we obtain Lemma \ref{boundaryharnackinequalty} for $1<p\leq 2$ when $\tilde{w}$ or $\hat{w}$ in $B(w,4r)\cap \partial\Omega$. 
Other values of $\tilde{w}, \hat{w}$ in (\ref{boundaryharnackinequaltyofu}) 
are handled by using this estimate and the interior estimate in Lemma \ref{uisholder}
\end{proof}
\begin{lemma}
\label{uismu}
For fixed $p$, $1<p<\infty$, let $\hat{u},\Omega, w, r$ be as in Lemma \ref{caccioppolitypeinequality}.
Let $\hat{\mu}$ be the measure corresponding to $\hat{u}$ as in (\ref{ast}).

Then
\begin{align}
\label{measureandu}
\tfrac{1}{c}\, r^{p-2} \hat{\mu}(B(w, \tfrac{r}{2})) \leq \esssup\limits_{B(w,r)} \hat{u}^{p-1} \leq c\, r^{p-2} \hat{\mu}(B(w, 2r)).
\end{align}
\end{lemma}
\begin{proof}
 A similar argument to the one in \cite{EL} can be applied to obtain Lemma \ref{uismu}.
\end{proof}

We next study the so called capacitary function. To this end, we choose $z_{0}\in\Omega$ and let $D=\Omega\setminus \overline{B(z_{0}, d(z_{0},\partial\Omega)/4)}$. Let $u$ 
be a capacitary function for $D$ relative to $f$. That is,
$u$ is a positive weak solution to (\ref{flaplace}) in $D$ with continuous boundary values, 
$u\equiv 0$ on $\partial\Omega$ and $u\equiv 1$ on $\partial B(z_{0}, d(z_{0},\partial\Omega)/4)$.
\begin{remark}
\label{compofmeasures}
It easily follows from Lemma \ref{uismu} that $\hat{\mu}, \mu$  corresponding to $\hat{u}, u$ 
respectively as in (\ref{ast}) are mutually absolutely continuous,
\[
 \hat{\mu} \ll \mu \ll \hat{\mu}.
\]
Hence  $\hd{\hat{\mu}}=\hd{\mu}$. We also conclude from mutual absolute continuity of $\hat{\mu}, \mu$ that Theorem \ref{maintheorem} holds for $\hat{\mu}$ if and only if it holds for $\mu$ (For a proof see \cite{LNP}).
\end{remark}
\subsection{More advanced regularity results}
\label{madvreg}
In this subsection we study more advanced regularity properties of a weak solution $\hat{u}$ to (\ref{flaplace}). 
We first obtain regularity results for $\nabla\hat{u}$. 
To this end, assume that $B(\hat{w},4r)\subset\Omega$ and let $\hat{u}_{\varepsilon}$ be a weak solution to
\begin{align}
\label{epsilonweaksoln}
 \begin{split}
  0&=\nabla\cdot(\nabla f_{\varepsilon}(\nabla \hat{u}_{\varepsilon}))=\sum\limits_{k=1}^{2}\frac{\urpartial}{\urpartial x_{k}}\left(\frac{\urpartial f_{\varepsilon}(\nabla \hat{u}_{\varepsilon})}{\urpartial\eta_k}\right)
 \end{split}
\end{align}
in  $B(\hat{w},2r)$ with $\hat{u}_{\varepsilon}-\hat{u}\in W^{1,p}_{0}(B(\hat{w},2r))$ where $f_{\varepsilon}=f\ast \theta_{\varepsilon}$ 
is as in (\ref{convprop}). 

Using the De Giorgi method, (\ref{epsilonstar}), and pde theory, it can be shown that  $\zeta=(\hat{u}_{\varepsilon})_{\xi}$ is in $W^{1,2}(B(\hat{w},r))$ and
satisfies a uniformly elliptic equation in divergence form (for more details see \cite{LU}). That is,  $\zeta=(\hat{u}_{\varepsilon})_{\xi}$ in $W^{1,2}(B(\hat{w},r))$
is a weak solution to
\begin{align}
 0=\sum\limits_{k,j=1}^{2}\frac{\urpartial}{\urpartial x_{k}}\left(\frac{\urpartial^{2}f^{\varepsilon}(\nabla \hat{u}_{\varepsilon})}{\urpartial \eta_{k}\urpartial\eta_{j}}\, \frac{\urpartial \zeta}{\urpartial x_{j}}\right).
\end{align}
in $B(\hat{w},2r)$. Here ellipticity constants and $W^{2,2}$ norm of $\hat{u}_{\varepsilon}$ depend on $\varepsilon$.
On the other hand, $\hat{u}_{\varepsilon}$ also satisfies a nondivergence form equation
\begin{align}
\label{nondivergencesolnu}
 \begin{split}
  0=\frac{1}{\left(|\nabla \hat{u}_{\varepsilon}|+\varepsilon\right)^{p-2}}\sum\limits_{j,k=1}^{2}\left(\frac{\urpartial^{2}f^{\varepsilon}(\nabla \hat{u}_{\varepsilon})}{\urpartial \eta_{k}\urpartial\eta_{j}}\right)\zeta_{x_{j}x_{k}}.
 \end{split}
\end{align}
in $B(\hat{w},2r)$. It follows from (\ref{epsilonstar}) that ellipticity constants are independent of $\varepsilon$. Using this fact and arguing
 as in \cite[Chapter 5]{GT} it follows that $(\hat{u}_{\varepsilon})_{z}=\hat{u}_{x_{1}}-\ri \hat{u}_{x_{2}}$ is a $K-$quasiregular mapping for some constant $K$ which depends on the constant $c$ in (\ref{epsilonstar}).

Then $\hat{u}_{\varepsilon}\in W^{2,2}(B(\hat{w},2r))$ with norm independent of $\varepsilon$. 
Also $\nabla \hat{u}_{\varepsilon}$ is $\alpha''-$H\"older continuous where $\alpha''=K-\sqrt{K^{2}-1}$ 
with constant independent of $\varepsilon$ (see \cite{AIM}).

Since $\nabla \hat{u}_{\varepsilon}\to \nabla \hat{u}$ in $W^{1,p}(B(\hat{w},2r))$, then for some subsequence, $\varepsilon^{i}\to 0$ we have
$\nabla \hat{u}_{\varepsilon_{i}}\to \nabla \hat{u}$ a.e in $B(\hat{w},2r)$. $\{\nabla \hat{u}_{\varepsilon_{i}}\}$ is 
equicontinuous as $\{\nabla \hat{u}_{\varepsilon_{i}}\}$ is uniformly H\"older continuous with constant 
independent of $\varepsilon$. We may redefine $\nabla \hat{u}$ in a set of measure zero if needed. Thus,
$\nabla \hat{u}_{\varepsilon_{i_{k}}}\to \nabla \hat{u}$ uniformly  on compact subsets of $B(\hat{w},2r)$.
Then it follows from \cite{AIM} that $\nabla \hat{u}$ is a $K-$quasiregular mapping. 

From quasiregularity we also have
\begin{align}
\|\nabla \hat{u}\|_{W^{1,2}(B(\hat{w},r)\cap\Omega)} \leq c \|\nabla \hat{u}\|_{L^{2}(B(\hat{w},\tfrac{3r}{2})\cap\Omega)}
\end{align}
where $c=c(p)$, and $\nabla \hat{u}$ is H\"older continuous. Using these facts and basic 
Cacciopoli type estimates for $\hat{u}_{\xi}$ we deduce the following lemma,
\begin{lemma}[Local interior regularity for $\nabla \hat{u}$]
\label{localholderfornablau}
Let $\hat{u}, f, \Omega, w$ be as in Lemma \ref{caccioppolitypeinequality}. 
If $B(\tilde{w}, 4s)\subset B(w,4r)\cap\Omega$, then $\hat{u}$ has a representative with H\"older continuous 
derivatives in $B(\tilde{w},2s)$ (also denoted $\hat{u}$). Moreover $\nabla\hat{u}$ is $K-$quasiregular and there exists $\alpha'''$, $0<\alpha'''<1$, and $c\geq 1$, depending only on $f$ and $p$, with
\begin{align}
 \label{nablauisholdercont}
\begin{split}
|\nabla \hat{u}(\tilde{z})-\nabla \hat{u}(\hat{z})| \leq c \left(\frac{|\tilde{z}-\hat{z}|}{s}\right)^{\alpha'''}\, \esssup\limits_{B(\tilde{w},s)}|\nabla \hat{u}| \leq 
\frac{c}{s}\left(\frac{|\tilde{z}-\hat{z}|}{s}\right)^{\alpha'''} \esssup\limits_{B(\tilde{w},s)}\hat{u}. 
\end{split}
\end{align}
Also if $\nabla \hat{u}\neq 0$ in $B(\tilde{w},2s)$, then
\begin{align}
\int\limits_{B(\tilde{w},s)}|\nabla \hat{u}|^{p-2}\sum\limits_{k,j=1}^{2}(\hat{u}_{x_{k}x_{j}})^{2}\rd\nu \leq \frac{c}{(t-s)^{2}}\int\limits_{B(\tilde{w},t)}|\nabla \hat{u}|^{p}\rd\nu. 
\end{align}
for $s<t<2s$.
\end{lemma}
\begin{lemma}
\label{logfgraduissoln}
Let $\hat{u}, f, \Omega, w, r$ be as in Lemma \ref{caccioppolitypeinequality}. 
If $\nabla \hat{u}\neq 0$ in 
$B(\tilde{w},4s)\subset B(w,4r)\cap\Omega$ then $h(z)=\log|\nabla \hat{u}|(z)$ is a weak solution to a uniformly elliptic divergence 
form partial differential equation for which a Harnack's inequality holds.
\end{lemma}
\begin{proof}
From Lemma \ref{localholderfornablau} $\nabla \hat{u}$ is a $K-$quasiregular mapping and 
by assumption $\nabla \hat{u}\neq 0$ in $B(\tilde{w},4s)$. Thus $h(z)$ is well-defined in 
$B(\tilde{w},4s)$ and also a weak solution to
 \begin{align}
 \label{logissolution}
  \sum\limits_{k,j=1}^{2}\frac{\partial}{\partial x_{k}}\left(\mathcal{A}_{kj}h_{x_{j}}\right)=0 \, \, \mathrm{in}\, \, B(\tilde{w},4s)
 \end{align}
where $(\mathcal{A}_{kj})=\mathcal{A}$, $D^{2}\hat{u}=\left(\frac{\urpartial^{2}\hat{u}}{\urpartial x_{k} \urpartial x_{j}}\right)$, and
\[
 \mathcal{A}=\left\{
\begin{array}{cc}
\det D^{2}\hat{u}\,\left(D^{2}\hat{u}^{\mbox{\tiny{T}}}\, D^{2}\hat{u}\right)^{-1} & \mathrm{if}\, \, D^{2}\hat{u}\ \mathrm{is\ invertible,}\\
\mbox{Identity matrix} & \mathrm{otherwise}
 \end{array}
\right.
\]
(for more details see \cite[Chapter 14]{HKM}). It follows from an observation in \cite[Theorem 14.61]{HKM} and $K-$quasiregularity of $\nabla \hat{u}$ that
\[
\tfrac{1}{c}|\eta|^{2}\leq \mathcal{A}\eta\cdot\eta\leq c |\eta|^{2}\, \ \mathrm{a.e}\, \, \mbox{in}\, \, B(\tilde{w},4s)\, \, \mbox{and for all}\, \, \eta\in\mathbb{R}^{2}.
\]
Therefore $h=\log|\nabla \hat{u}|$ is a weak solution to a uniformly elliptic partial differential equation in divergence form in $B(\tilde{w},4s)$
 from which we conclude that Harnack's inequality can be applied to $h$ in $B(\tilde{w},4s)$ when $h>0$.
\end{proof}
\begin{lemma}
\label{nablaisnotzeroinomega}
Let $u$ be the capacitary function for $D$ defined after Lemma \ref{uismu}. Then $\nabla u\neq 0$ in $D$.
\end{lemma}
\begin{proof}
Since (\ref{flaplace}) is invariant under dilation and translation 
we may assume that $D=\Omega\setminus \overline{B(0,1)}$. To prove Lemma \ref{nablaisnotzeroinomega} we use the principle of the argument. 
Indeed, we use the principle of the argument for a $K-$quasiregular mapping.

Let $u_{z}=u_{x_{1}}-\ri u_{x_{2}}$. We know from the previous subsection that $u_{z}$ is a non constant $K-$quasiregular 
mapping and therefore that the zeros of $u_{z}$ are isolated and countable in $D$. 
Hence, there exist $0<t_{0}<t_{1}<1$ with
 $t_{0}$ arbitrarily close to $0$ and $t_{1}$ arbitrarily close to $1$ such that
 $u_{z} \neq 0$ on $\gamma_{j}=\{z\in D;\,\, u(z)=t_{j}\}$ for $j=0,1$. $K-$quasiregularity of $u_{z}$ implies
that $u_{z}$ is $\alpha'''-$H\"older continuous for some $0<\alpha'''<1$. Then from Lemma \ref{localholderfornablau} $\gamma_{j}$, $j=0,1$, is a $C^{1,\alpha'''}$ 
Jordan curve and without loss of generality we can assume that
$\gamma_{j}$ is oriented counterclockwise for $j=0,1$.

Let $\Gamma_{j}=u_{z}(\gamma_{j})$ for $j=0,1$. We claim that 
\begin{align}
\label{principleofargument}
\frac{1}{2\pi\ri}\left(\int\limits_{\Gamma_{0}}\frac{\rd w}{w}-\int\limits_{\Gamma_{1}}\frac{\rd w}{w}\right)=\#\, \mathrm{of\ zeros\ of}\, \, u_{z}(z)\, \mathrm{in}\, \, \{z\in D;\, t_{0}<u(z)<t_{1}\}	
\end{align}
Indeed, (\ref{principleofargument}) is well-known if $u_{z}$ is an analytic function as follows from the "principle of the argument".

We prove (\ref{principleofargument}) using this idea and the Sto\"ilov factorization theorem, that is
\begin{align}
\label{uishg}
u_{z}(z)=h\circ g(z), \, \, z\in D
\end{align}
where $h$ is an analytic function in $g(D)$ and $g$ is a $K-$quasiconformal mapping of $D$. Then
\begin{align}
 \partial g(\{z\in D;\, t_{0}<u(z)<t_{1}\}) = \tau_{0}\cup \tau_{1}=(g\circ\gamma_{0})\cup (g\circ\gamma_{1})
\end{align}
where $\tau_{j}=g\circ\gamma_{j}$ is a $C^{\beta}$ Jordan curve for some $0<\beta<1$, oriented counterclockwise for $j=0,1$. 
Applying the principle of the argument to $h$ as in (\ref{principleofargument})
we get
\begin{align}
\label{principleofargumentforf}
 \frac{1}{2\pi\ri}\left[\triangle\, \mathrm{arg}\, (h\circ\tau_{0})-\triangle\, \mathrm{arg}\, (h\circ\tau_{1})\right]=\#\, \mathrm{of\ zeros\ of}\, \, h\, \, \mathrm{in}\,\, g(\{z\in D;\, t_{0}<|z|<t_{1}\}).
\end{align}
Here $\triangle\, \mathrm{arg}\, (h\circ\tau_{j})$, $j=0,1$, denotes the change in the argument of $h\circ\tau_{j}$ as $\tau_{j}$ 
is traversed counterclockwise. (\ref{principleofargument}) follows
from the fact that $g^{-1}$ is a homeomorphism of $\mathbb{C}$ onto $\mathbb{C}$ and (\ref{principleofargumentforf})(See \cite{AIM}).

Now, let $z_{j}(s)$, $0\leq s \leq 1$ be a parametrization of $\gamma_{j}$ for $j=0,1$. Since $\gamma_{j}$ is $C^{1,\alpha}$ we have
\begin{align}
\begin{split}
 0=\frac{\rd }{\rd s} (t_{j})&=\frac{\rd }{\rd s}(u(z_{j}(s)))\\
			    &= u_{z}\,\frac{\rd z_{j}(s)}{\rd s}+ u_{\overline{z}}\, \frac{\rd \overline{z_{j}}(s)}{\rd s}\\
			    &=2\, \mathrm{Re}[u_{z}\, \frac{\rd z_{j}(s)}{\rd s}].
\end{split}
\end{align}
Therefore, $u_{z}\frac{\rd z_{j}(s)}{\rd s}$ is always pure imaginary on $\gamma_{j}$, $j=0,1$, and so
\begin{align}
\label{sumofarguments}
\begin{split}
0&=\triangle\, \mathrm{arg}\, [u_{z}\, \frac{\rd z_{j}(s)}{\rd s}] \\
&=\triangle\, \mathrm{arg}\, u_{z}(\gamma_{j})+\triangle\, \mathrm{arg}\, \frac{\rd z_{j}(s)}{\rd s}.
\end{split}
\end{align}
From (\ref{sumofarguments}) we see that
\begin{align}
\label{arguzisargubarz}
\triangle\, \mathrm{arg}\, u_{z}(\gamma_{j})=-\triangle\, \mathrm{arg}\, \frac{\rd z_{j}(s)}{\rd s}
\end{align}

Finally, as $\gamma_{j}$, $j=0,1$ is a Jordan curve oriented counterclockwise, it follows from the Gauss-Bonnet Theorem that
\begin{align}
 \label{gaussbonnet}
 \frac{1}{2\pi}\triangle\, \mathrm{arg}\, \frac{\rd z_{j}}{\rd s}=1\,\, \mbox{for}\, \, j=1,2.
\end{align}
Another way to prove (\ref{gaussbonnet}) using analytic function theory is to use the Riemann mapping theorem to first get $\psi_{j}$ 
mapping $\{z\;\, |z|<1\}$ 
onto $G_{j}=:$ inside of $\gamma_{j}$, $j=0,1$. As in \cite{T} it follows that $\psi_{j}$ extends to a $C^{1,\beta}$ 
homeomorphism of $\{z\;\, |z|\leq 1\}$
onto $\overline{G_{j}}$. Then we can put 
\begin{align}
 z_{j}(s)=\psi_{j}(e^{2\pi\ri s}),\, \, 0\leq s \leq 1,
\end{align}
and observe that
\begin{align}
 \frac{\rd z_{j}(s)}{\rd s}=2\pi\ri \psi'_{j}(e^{2\pi\ri s})e^{2\pi\ri s}.
\end{align}
Then on $\{z;\, |z|=1\}$ we have
\begin{align}
\label{sumis1}
\begin{split}
 \triangle\, \mathrm{arg}\, \frac{\rd z}{\rd s} &= \triangle\, \mathrm{arg}\, \psi_{j}'(z)+\triangle\, \mathrm{arg}\, z\\
						 &=0+2\pi=2\pi.
\end{split}
\end{align}
In view of (\ref{uishg}), (\ref{principleofargumentforf}), (\ref{arguzisargubarz}), and (\ref{gaussbonnet}) we conclude
$u_{z}\neq 0$ in $G_{1}\setminus G_{0}$, i.e between the level sets $\gamma_{0}$ and $\gamma_{1}$. 
Using this observation and letting $t_{0}\to 0$, $t_{1}\to 1$ in (\ref{principleofargument}) we have the desired result, $u_{z}\neq 0$ in $D$.
\end{proof}
Next we state the fundamental inequality from \cite{LNP}.
\begin{lemma}
\label{finequality}
Let $u$ be a capacitary function defined after Lemma \ref{uismu} for 
$D=\Omega\setminus \overline{B(z_{0}, d(z_{0},\partial\Omega)/4)}$. Then there is a constant $c=c(f,p)$ such that 
\begin{align}
\tfrac{1}{c}\frac{u(z)}{d(z,\partial\Omega)}\leq |\nabla u(z)| \leq c \frac{u(z)}{d(z,\partial\Omega)} 
\end{align}
whenever $z\in D$ and $d(z,\partial\Omega)\geq \frac{d(z_{0},\partial\Omega)}{2}$.
\end{lemma}
\begin{proof}
Fix $p$, $1<p<\infty$, and let $u$ Let $u$ be a capacitary function defined after Lemma \ref{uismu} for 
$D=\Omega\setminus \overline{B(z_{0}, d(z_{0},\partial\Omega)/4)}$. The proof in \cite{LNP} uses only Harnack's inequality 
for a $p-$harmonic function and H\"older continuity of $u$ as well as Harnack's inequality for $\log|\nabla u|$ when $\nabla u\neq 0$. Since our 
function $u$ has these properties we conclude that (\ref{finequality}) is also valid in our situation (For more details see \cite[Theorem 1.5]{LNP}).
\end{proof}
 
\begin{lemma}
\label{loguisweaksoln}
Let $u$ be a capacitary function for $D$ defined after Lemma \ref{uismu} and let $f$ be as in 
Theorem \ref{maintheorem}. Then $v=\log f(\nabla u)$ is a weak sub solution, solution or super
 solution to $L\zeta=0$ respectively when $2<p<\infty$, $p=2$ or $1<p<2$. where
\begin{align}
\label{Lzeta0}
L\zeta\, = \, \sum\limits_{k,j=1}^{2}\frac{\urpartial }{\urpartial x_{k}}\left(f_{\eta_{j}\eta_{k}}(\nabla u) \frac{\urpartial \zeta}{\urpartial x_{j}}\right).
\end{align}
\end{lemma}
\begin{remark}
When $f$ in (\ref{flaplace}) is smooth enough and homogeneous of degree p in $\mathbb{C}\setminus\{0\}$ and $u$
is smooth enough as well as a pointwise solution to (\ref{Lzeta0}), then in \cite[Theorem 1]{ALV} it is shown by a direct calculation that $\log f(\nabla u)$ is a sub solution, 
solution or super solution to the partial differential equation in (\ref{Lzeta0}) respectively when $2<p<\infty$, $p=2$ or $1<p<2$.
\end{remark}
\begin{proof}[Proof of Lemma \ref{loguisweaksoln}]
Using Lemmas \ref{localholderfornablau}, \ref{finequality} in (\ref{flaplace}) with $u'=u$, we find for $\phi\in C^{\infty}_{0}(D)$ that
\begin{align}
 \label{intbypartszlsoln}
\begin{split}
0&=\int\limits_{\Omega}\langle \nabla f(\nabla u), \nabla \phi_{x_{l}}\rangle \rd\nu
=-\int\limits_{\Omega}\sum\limits_{k=1}^{2}  \frac{\urpartial (f_{\eta_{k}}(\nabla u))}{\urpartial x_{l}}\, \phi_{x_{k}} \rd\nu \\
&=-\int\limits_{\Omega}\sum\limits_{k,j=1}^{2}  f_{\eta_{k}\eta_{j}}(\nabla u))(u_{x_{l}})_{x_{j}} \phi_{x_{k}} \rd\nu.
\end{split}
\end{align} 
From homogeneity of $f$ and Euler's formula we have
\begin{align}
\label{homogenandeuler}
\sum\limits_{j=1}^{2}\eta_{j}f_{\eta_{k}\eta_{j}}(\eta)=(p-1)f_{\eta_{k}}(\eta)\, \, \, \mathrm{and}\, \, \, \sum\limits_{j=1}^{2} \eta_{j}f_{\eta_{j}}(\eta)=pf(\eta) 
\end{align}
for $k=1,2$ and for a.e. $\eta$. Then it follows from (\ref{intbypartszlsoln}) and (\ref{homogenandeuler}) that
\begin{align}
\label{Luiszero}
\begin{split}
\int\limits_{\Omega}\sum\limits_{k,j=1}^{2}f_{\eta_{j}\eta_{k}}(\nabla u)u_{x_{j}} \phi_{x_{k}}\, \rd\nu=(p-1)\int\limits_{\Omega}\sum\limits_{k=1}^{2} f_{\eta_{k}}(\nabla u) \phi_{x_{k}}\rd\nu=0.
\end{split}
\end{align}
From (\ref{Luiszero}) we see that $\zeta=u$ is also a weak solution to $L\zeta=0$. We note also that since $u,f\in W^{2,2}_{\mbox{\tiny{loc}}}(D)$ thanks to
Lemmas \ref{fishomdp}, \ref{localholderfornablau}, and \ref{finequality} then for $\nu$ a.e $z\in \Omega$ 
\begin{align}
 \label{aezero}
0=\sum\limits_{k,l=1}^{2}f_{\eta_{k}\eta_{l}}(\nabla u(z))u_{x_{k}x_{l}}(z).
\end{align}
Let $v=\log f(\nabla u)$, $\mathfrak{b}_{ij}=f_{\eta_{i}\eta_{j}}(\nabla u)$, $D^{2}u=(u_{x_{i}x_{j}})$, $D^{2}f=(f_{\eta_{i}\eta_{j}})$ and observe that
\begin{align}
 \label{bijvij}
\mathfrak{b}_{kj}v_{x_{j}}=\frac{1}{f(\nabla u)}\sum\limits_{n=1}^{2}f_{\eta_{n}}(\nabla u)\mathfrak{b}_{kj}u_{x_{n}x_{j}}.
\end{align}
Using (\ref{bijvij}) we see that
\begin{align}
\label{IandII}
 \begin{split}
\int\limits_{\Omega}\sum\limits_{k,j=1}^{2}\mathfrak{b}_{kj}v_{x_{j}}\phi_{x_{k}}\, \rd\nu
&=\int\limits_{\Omega}\sum\limits_{k,j=1}^{2}\frac{1}{f(\nabla u)}\sum\limits_{n=1}^{2}\mathfrak{b}_{kj}f_{\eta_{n}}(\nabla u)u_{x_{n}x_{j}}\phi_{x_{k}}\, \rd\nu\\
&=-\int\limits_{\Omega}\sum\limits_{n,k,j=1}^{2}\frac{\urpartial }{\urpartial x_{k}}\left(\frac{f_{\eta_{n}}(\nabla u)}{f(\nabla u)}\right)\mathfrak{b}_{kj}u_{x_{n}x_{j}}\, \phi\, \rd\nu
\end{split}
\end{align}
where to get the last line in (\ref{IandII}) we have used
\begin{align}
 \label{IandIIcomp}
0=\int\limits_{\Omega}\sum\limits_{n,k,j=1}^{2}\mathfrak{b}_{kj}u_{x_{n}x_{j}}\frac{\urpartial }{\urpartial x_{k}}\left(\frac{f_{\eta_{n}}(\nabla u)}{f(\nabla u)}\phi\right)\rd\nu.
\end{align}
(\ref{IandIIcomp}) is a consequence of (\ref{intbypartszlsoln}) with $n=l$ and $\phi$ replaced by $\frac{f_{\eta_{n}}(\nabla u)}{f(\nabla u)}\phi$ as well as the fact that 
\[
\frac{f_{\eta_{n}}(\nabla u)}{f(\nabla u)}\in W^{1,2}_{\mbox{\tiny{loc}}}(D).
\]
From (\ref{IandII}) we have
\begin{align}
 \label{IandIIimp}
\begin{split}
\int\limits_{\Omega}\sum\limits_{k,j=1}^{2}\mathfrak{b}_{kj}v_{x_{j}} \phi_{x_{k}}\, \rd\nu&=-\int\limits_{\Omega}\sum\limits_{n,k,j=1}^{2}\frac{\urpartial }{\urpartial x_{k}}\left(\frac{f_{\eta_{n}}(\nabla u)}{f(\nabla u)}\right)\mathfrak{b}_{kj}u_{x_{n}x_{j}}\, \phi\, \rd\nu \\
&=-\int\limits_{\Omega}(I'+I'')\phi\rd\nu 
\end{split}
\end{align}
where
\begin{align}
\label{I'}
I'=\sum\limits_{n,j,k,l=1}^{2}\frac{1}{f(\nabla u)}\, \mathfrak{b}_{nl}\mathfrak{b}_{kj}u_{x_{l}x_{k}}u_{x_{n}x_{j}}
\end{align}
and
\begin{align}
\label{I''}
I''=-\frac{1}{f^{2}(\nabla u)}\,\sum\limits_{n,j,k,l=1}^{2}\mathfrak{b}_{kj}f_{\eta_{n}}(\nabla u)f_{\eta_{l}}(\nabla u)u_{x_{l}x_{k}}u_{x_{n}x_{j}}.
\end{align}
We can rewrite (\ref{I'}) and (\ref{I''}) using matrix notation. First notice that (\ref{aezero}) becomes
\begin{align}
\label{d2fd2uw}
\mathrm{tr}\left(D^{2}f \cdot D^{2}u\right)=0\, \, \mbox{for}\, \, \nu\, \, \mbox{a.e}\, \, z\, \, \mbox{in}\, \, D.
\end{align}
It follows from (\ref{d2fd2uw}) that there exists $\mathfrak{m}, \mathfrak{n}, \mathfrak{l}$ such that
\begin{align}
\label{traceD}
D^{2}f \cdot D^{2}u=\begin{bmatrix} \mathfrak{m} & \mathfrak{n} \\ \mathfrak{l} & -\mathfrak{m} \end{bmatrix}
\, \, \mbox{for}\,\, \nu\, \, \mbox{a.e}\, \, z\, \, \mbox{in}\, \, D.
\end{align}
Squaring both sides of (\ref{traceD}) gives that
\begin{align}
\label{traceofop}
 (D^2f \cdot D^2u)^2 = (\mathfrak{m}^2+\mathfrak{n}\mathfrak{l})\begin{bmatrix} 1 & 0 \\ 0 & 1\end{bmatrix} = -\det(D^2f \cdot D^2u)I \, \, \mbox{for}\, \, \nu\, \, \mbox{a.e}\, \, z \, \, \mbox{in}\, \, \Omega.
\end{align}
Using (\ref{traceD}) and (\ref{traceofop}) we can write (\ref{I'}) as
\begin{align}
\label{I'inmatrix}
\begin{split}
I'&=\frac{1}{f}\,\text{tr}\left((D^{2}f \cdot D^{2}u)^{2}\right)\\
&=-\frac{\det(D^{2}f\cdot D^{2}u)}{f}\, \text{tr}(I)\\
&=-2\frac{\det(D^{2}f\cdot D^{2}u)}{f}.
\end{split}
\end{align}
To handle (\ref{I''}) note from symmetry of $D^{2}u$ and $D^{2}f$ that
\[
 \sum\limits_{k,j=1}^{2}\mathfrak{b}_{kj}u_{x_{l}x_{k}}u_{x_{n}x_{j}}
\]
is the $ln$ element of $D^{2}u\cdot D^{2}f\cdot D^{2}u$. 

Using homogeneity of $f$ for $\nu$ a.e $z$ in $D$ we obtain
\begin{align}
\label{I''inmatrix}
\begin{split}
I''&=-\frac{1}{f^{2}}\,\text{tr}\left(Df \cdot (Df)^{\mbox{\tiny{T}}}\cdot D^{2}u \cdot D^{2}f \cdot D^2 u\right) \\
&=-\frac{1}{f^{2}}\,\text{tr}\left(\tfrac{1}{(p-1)^{2}}\, D^{2}f \cdot \nabla u \cdot (D^{2}f \cdot \nabla u)^{\mbox{\tiny{T}}} \cdot D^2u \cdot D^2f \cdot D^2u \right) \\
&=-\frac{1}{f^{2}}\,\text{tr}\left(\tfrac{1}{(p-1)^{2}}D^{2}f \cdot \nabla u \cdot (\nabla u)^{\mbox{\tiny{T}}} \cdot \left(D^2f \cdot D^2u\right)^{2}\right)\\
&=\tfrac{1}{(p-1)^{2}}\frac{\det(D^2f \cdot D^2u)}{f^{2}}\,\text{tr}\left(D^{2}f \cdot \nabla u \cdot (\nabla u)^{\mbox{\tiny{T}}} \right) \\
&=\tfrac{1}{(p-1)^{2}}\frac{\det(D^2f \cdot D^2u)}{f^{2}}\, (\nabla u)^{\mbox{\tiny{T}}} \cdot D^{2}f \cdot  \nabla u\\
&=\tfrac{p(p-1)}{(p-1)^{2}}\frac{f\, \det(D^2f \cdot D^2u)}{f^{2}}\\
&=\tfrac{p}{(p-1)}\frac{\det(D^2f \cdot D^2u)}{f}
\end{split}
\end{align}
where we have used
\begin{align}
\label{squarematrix}
\text{tr}(D^2f \cdot \nabla u \cdot (\nabla u)^{\mbox{\tiny{T}}})=\sum\limits_{l,k=1}^2 \mathfrak{b}_{lk}u_{x_{l}} u_{x_{k}} =(\nabla u)^{\mbox{\tiny{T}}} \cdot D^2f \cdot \nabla u.
\end{align}
 Note that (\ref{I'inmatrix}) and (\ref{I''inmatrix}) imply for $\nu$ a.e $z$ in $D$
\begin{align}
\label{I'plusI''}
\begin{split}
I'+I''&=-2\frac{\det(D^{2}f\cdot D^{2}u)}{f}+\frac{p}{(p-1)}\, \frac{\det(D^2f \cdot D^2u)}{f} \\
      &=-\left(\frac{p-2}{p-1}\right)\, \frac{\det(D^2f \cdot D^2u)}{f}.
\end{split}
\end{align}
Rearranging (\ref{aezero}) for $\nu$ a.e $z$ in $D$  and using Lemma \ref{fishomdp} we find that 
\begin{align}
\label{ud2fut}
\begin{split}
-\frac{\mathfrak{b}_{11}}{|\nabla u|^{p-2}}\, \det(D^2u)&=\frac{1}{|\nabla u|^{p-2}}\left((2\mathfrak{b}_{12}u_{x_{1}x_{2}}+\mathfrak{b}_{22}u_{x_{2}x_{2}})u_{x_{2}x_{2}}+\mathfrak{b}_{11}u_{x_{1}x_{2}}^2\right) \\
&=\frac{1}{|\nabla u|^{p-2}}\left((\nabla u_{x_{2}})^{\mbox{\tiny{T}}} \cdot D^2f \cdot \nabla u_{x_{2}}\right) \approx |\nabla u_{x_{2}}|^{2}.
\end{split}
\end{align}
Likewise,
\begin{align}
\label{ud2futb22}
\begin{split}
-\frac{\mathfrak{b}_{22}}{|\nabla u|^{p-2}}\, \det(D^2u)=\frac{1}{|\nabla u|^{p-2}}\left((\nabla u_{x_{1}})^{\mbox{\tiny{T}}} \cdot D^2f \cdot \nabla u_{x_{1}}\right)\approx |\nabla u_{x_{1}}|^{2}.
\end{split}
\end{align}
Now from (\ref{IandIIimp}) and (\ref{I'plusI''}) we see that
\begin{align}
\label{combd2ud2fpos}
\begin{split}
\int\limits_{\Omega}\sum\limits_{k,j=1}^{2}\mathfrak{b}_{kj}v_{x_{j}}\phi_{x_{k}}\, \rd\nu &=\left(\frac{p-2}{p-1}\right) \int\limits_{\Omega} 
\frac{\det(D^{2}u\cdot D^{2}f)}{f}\, \phi\, \rd\nu \\
&\approx -\left(\frac{p-2}{p-1}\right) \int\limits_{\Omega}\frac{-\mathfrak{b}_{11}\det(D^{2}u)}{|\nabla u|^{p-2}}\, \frac{\det(D^{2}f)}{f}\, \phi\, \rd\nu.
\end{split}
\end{align}
Moreover, we note that combining (\ref{ud2fut}), (\ref{ud2futb22}) and using (\ref{combd2ud2fpos}), Lemma \ref{fishomdp} we have
\begin{align}
\label{finalforLweakly}
 Lv=(p-2)\mathcal{F} \, \, \, \mbox{weakly}
\end{align}
where $\mathcal{F}\approx |\nabla u|^{p-4}\sum\limits_{i,j=1}^{2}(u_{x_{i}x_{j}})^{2}$. From (\ref{finalforLweakly}) we conclude
for $p=2$ that we have $\zeta=v=\log(f(\nabla u))$ is a weak solution to (\ref{Lzeta0}), $L\zeta=0$. Similarly, $\zeta=v$ is a weak sub solution or super solution
to $L\zeta=0$ respectively when $2<p<\infty$ or $1<p<2$.
\end{proof}

\section{Proof of Theorem \ref{maintheorem}}
\label{proofofmain}
In this section, we first obtain Lemma \ref{mainlemma}, and then using this lemma we prove Theorem \ref{maintheorem} for 
fixed $p$ when $1<p\leq 2$ and $2\leq p<\infty$ separately. 
To this end, we shall give the definitions of $u, \Omega, z_{0}, \mu$ again. 

Let $\Omega$ be a bounded simply connected domain in the plane.
Let $z_{0}\in\Omega$ and let $D=\Omega\setminus \overline{B(z_{0}, d(z_{0},\partial\Omega)/4)}$. Let $u$ 
be a capacitary function for $D$. That is,
$u$ is a positive weak solution to (\ref{flaplace}) in $D$ with continuous boundary values, 
$u\equiv 0$ on $\partial\Omega$ and $u\equiv 1$ on $\partial B(z_{0}, d(z_{0},\partial\Omega)/4)$.
Then by Remark \ref{compofmeasures} we have $\hd{\hat{\mu}}=\hd{\mu}$. Therefore, it suffices to prove Theorem \ref{maintheorem} when $u$ is a 
capacitary function and $\mu$ is the measure corresponding to $u$ as in (\ref{ast}).

Let $D$ be as above and let $4\tilde{s}=d(z_{0},\partial\Omega)$ and set $\Xi(z)=z_{0}+\hat{s}z$. 
Then it follows from the fact that (\ref{flaplace}) is invariant 
under translation and dilation that $\tilde{u}=u(\Xi(z))$ for $\Xi(z)\in D$ is also a weak solution to (\ref{flaplace}) in $\Xi^{-1}(D)$. 
Let $\tilde{\mu}$ be the measure corresponding to $\tilde{u}$ in (\ref{ast}). It can be easily shown from (\ref{flaplace}) that
\begin{align}
\label{transdilat}
 \tilde{\mu}(E)=\hat{s}^{p-2}\mu(\Xi(E))\,\,\, \mbox{whenever}\,\, \, E\subset\mathbb{R}^{2}\,\, \mbox{is a Borel set}. 
\end{align}
Clearly, (\ref{transdilat}) implies that $\hd{\tilde{\mu}}=\hd{\mu}$. 
Therefore without loss of generality we can assume that $z_{0}=0$ and $d(z_{0},\partial\Omega)=4$, $D=\Omega\setminus\overline{B(0,1)}$. 

To prove Theorem \ref{maintheorem} we first need a lemma. To this end, let $u$ be a capacitary function for $D=\Omega\setminus \overline{B(0,1)}$ corresponding to $f$, and let $\mu$ be the corresponding Borel measure.
Define
\[
w(z)=\left\{
\begin{array}{ll}
\max(v(z),0) & \mbox{when}\, \, 1<p<2 \\ 
\max(-v(z),0) & \mbox{when}\, \, 2<p<\infty \\
\end{array}
\right.
\]
for $z\in D$ where $v(z)=\log(f(\nabla u)(z))$.
\begin{lemma}
\label{mainlemma}
Let $m$ be a nonnegative integer. There exists $c_{\ast}=c_{\ast}(f,p)\geq 1$ such that for
$0<t<1/2$,
\begin{align}
\label{lemma9}
\begin{split}
 \int\limits_{\{z\in D:\, \, u(z)=t\}} \frac{f(\nabla u)}{|\nabla u|}w^{2m}dH^{1}(z) \leq c_{\ast}^{m+1}m![\log\frac{1}{t}]^{m}.
\end{split}
\end{align}
\end{lemma}
\begin{proof}
Define  $g(z)=\max(w(z)-c',0)$, $z\in D$ where $c'$ is large enough so that $g\equiv 0$ in
$B(0,2)\cap D$. Since $u$ is continuous in $D$, there is such a $c'$.

Extend $g$ continuously to $\Omega$ by putting  $g \equiv 0$ in $\overline{B(0,1)}$.
Set $\mathfrak{b}_{ij}=f_{\eta_{i}\eta_{j}}(\nabla u)$ and let $L$ be as in Lemma \ref{loguisweaksoln}.

Let $\Omega(t)= \{z\in D:\, u(z)> t\}$ for $0<t<1/2$ and let $\tilde{u}=\max(u-t,0)$. Note that $g^{2}\in W^{2,\infty}(\Omega(t))$.

Fix $p$, $1<p\leq 2$ until further notice. From Lemma \ref{loguisweaksoln}, $\zeta=v=\log f(\nabla u)$ 
is a weak super solution to $L\zeta=0$ in $D$ (see (\ref{Lzeta0})). Using 
$g^{2m-1}\tilde{u}\geq 0$ as a test function in (\ref{Lzeta0}) for $\zeta=g$ and the fact that $g\equiv 0$ in $B(0,2)$, we get
\begin{align}
\label{II'+II''}
 \begin{split}
0 &\leq 2m\int\limits_{\Omega(t)}\sum\limits_{k,j=1}^{2}\mathfrak{b}_{kj}\frac{\urpartial}{\urpartial x_{j}}(\log f(\nabla u))\frac{\urpartial }{\urpartial x_{k}}\left(g^{2m-1} \tilde{u}\right)\rd\nu \\
  &=2m\int\limits_{\Omega(t)}\sum\limits_{k,j=1}^{2}\mathfrak{b}_{kj}g_{x_{j}}\frac{\urpartial }{\urpartial x_{k}}\left(g^{2m-1}(u-t)\right)\rd\nu \\
&=2m(2m-1)\int\limits_{\Omega(t)}\sum\limits_{k,j=1}^{2}\mathfrak{b}_{kj}g_{x_{j}}g_{x_{k}}g^{2m-2}(u-t)\rd\nu \\
&\, \, +2m\int\limits_{\Omega(t)}\sum\limits_{k,j=1}^{2}\mathfrak{b}_{kj}g_{x_{j}}g^{2m-1}(u-t)_{x_{k}}\rd\nu\\
	&=II'+II''.
\end{split}
\end{align}
We first handle $II''$. To this end, let $\psi\in C^{\infty}_{0}(\{z:\, \, u(z)>t-\varepsilon\})$ with $\psi=1$ on $\overline{\Omega(t)}$.
Then since $\zeta=u$ is a weak solution to (\ref{Lzeta0}) and using $g^{2m}\psi$ as a test function, we obtain
\begin{align}
\label{II''}
 \begin{split}
0&=\int\limits_{\Omega(t-\varepsilon)}\sum\limits_{k,j=1}^{2} \mathfrak{b}_{kj} u_{x_{k}}\frac{\urpartial }{\urpartial x_{j}}\left(\psi g^{2m}\right)\rd\nu \\
&=2m\int\limits_{\Omega(t-\varepsilon)}\sum\limits_{k,j=1}^{2}\mathfrak{b}_{kj} u_{x_{k}} g^{2m-1}g_{x_{j}}\psi\rd\nu+\int\limits_{\Omega(t-\varepsilon)}\sum\limits_{k,j=1}^{2}\mathfrak{b}_{kj} u_{x_{k}} g^{2m}\psi_{x_{j}}\rd\nu\\
&=II''_{1}+II''_{2}.
 \end{split}
\end{align}
Letting $\varepsilon\to 0$ and using the Lebesgue dominated convergence theorem gives $II''_{1}\to II''$. 

We now show that for $H^{1}$ a.e $t\in (0, 1/2)$ and properly chosen $\psi$ that
\begin{align}
\label{II''2}
 II''_{2}\to \int\limits_{\{z\in D:\, \, u(z)=t\}} \sum\limits_{k,j=1}^{2}\mathfrak{b}_{kj}g^{2m}u_{x_{k}} \frac{u_{x_{j}}}{|\nabla \tilde{u}|}\, \,\, \mbox{as}\, \, \varepsilon\to 0.
\end{align}
To this end let $\phi:\mathbb{R}\to \mathbb{R}$ be a $C^{\infty}$ function satisfying $0\leq \phi \leq 1$, and $|\phi'|\leq c/\varepsilon$ such that
\[
\phi(s)=\left\{
\begin{array}{ll}
1 & \mbox{when}\, \, s\geq 1,\\ 
0 & \mbox{when}\, \, s\leq 1-\varepsilon.
\end{array}
\right.
\]
If we set $\psi=\phi(u(z)/t)$ in $II''_{2}$ and use the coarea formula we see that
\begin{align}
\label{II''22goesto}
 \begin{split}
II''_{2}&=\int\limits_{\Omega(t-\varepsilon)}\sum\limits_{k,j=1}^{2}\mathfrak{b}_{kj} u_{x_{k}} g^{2m}\psi_{x_{j}}\rd\nu\\
	&=\int\limits_{\Omega(t(1-\varepsilon))}\sum\limits_{k,j=1}^{2}\mathfrak{b}_{kj} u_{x_{k}} g^{2m}\left(\phi\left(\frac{u(z)}{t}\right)\right)_{x_{j}}\rd\nu\\
	&=\frac{1}{t}\int\limits_{\Omega(t(1-\varepsilon))}\sum\limits_{k,j=1}^{2}\mathfrak{b}_{kj} u_{x_{k}} g^{2m}\phi'\left(\frac{u(z)}{t}\right)u_{x_{j}}\rd\nu\\
	&=\frac{1}{t}\int\limits_{t(1-\varepsilon)}^{t}\phi'(\frac{\tau}{t})\left(\int\limits_{\{z\in D:\, \, u(z)=\tau\}}\sum\limits_{k,j=1}^{2}\mathfrak{b}_{kj} u_{x_{k}}g^{2m}\frac{u_{x_{j}}}{|\nabla u|}\rd H^{1}\right)\rd\tau.
\end{split}
\end{align}
Let 
\[
\Theta(\tau)=\int\limits_{\{z\in D:\,\, u(z)=\tau\}}\sum\limits_{k,j=1}^{2}\mathfrak{b}_{kj} u_{x_{k}}g^{2m}\frac{u_{x_{j}}}{|\nabla u|}\rd H^{1}.
\]
Then using
\[
 \frac{1}{t}\int\limits_{t(1-\varepsilon)}^{t}\phi'\left(\frac{\tau}{t}\right)\rd\tau=\phi(1)-\phi(1-\varepsilon)=1
\]
we have 
\begin{align}
\begin{split}
 II''_{2}=\frac{1}{t}\int\limits_{t(1-\varepsilon)}^{t}\phi'(\frac{\tau}{t})\left[\Theta(\tau)-\Theta(t)\right]\rd\tau + \Theta(t)
\end{split}
 \end{align}
for almost every $t\in (0,1/2)$.
If we let $\varepsilon \to 0$ it follows from the strong form of the Lebesgue Differentiation theorem that
\begin{align}
\label{sendeto0}
 \lim\limits_{\varepsilon\to 0}|\frac{1}{t}\int\limits_{t(1-\varepsilon)}^{t}\phi'(\frac{\tau}{t})\left[\Theta(\tau)-\Theta(t)\right]\rd\tau| &\leq\lim\limits_{\varepsilon\to 0} \frac{1}{t\, \varepsilon}\int\limits_{t(1-\varepsilon)}^{t}|\Theta(\tau)-\Theta(t)|\rd\tau=0														
\end{align}
for $H^{1}$ a.e $t\in (0,1/2)$. From (\ref{II''22goesto}) and (\ref{sendeto0}) for $H^{1}$ a.e $t\in (0,1/2)$ we have
\begin{align}
\label{II''2to0}
 \begin{split}
  II''_{2}\to \Theta(t)\, \, \mbox{as}\,\, \varepsilon\to 0.
 \end{split}
\end{align}
Thus (\ref{II''2}) is true. Hence using (\ref{II''2to0}) in (\ref{II''2}) and then (\ref{II''}) and (\ref{II''2}) in (\ref{II'+II''}) we see that
\begin{align}
\label{II''combined}
\begin{split}
\int\limits_{\{z:\,u(z)=t\}} \sum\limits_{k,j=1}^{2} & \mathfrak{b}_{kj}g^{2m}u_{x_{k}} \frac{u_{x_{j}}}{|\nabla u|} dH^{1}(z) \\
&\leq 2m(2m-1)\int\limits_{\Omega(t)}\sum\limits_{k,j=1}^{2}\mathfrak{b}_{kj}g_{x_{j}}g_{x_{k}}g^{2m-2}(u-t)\rd\nu
\end{split}
\end{align}

Similarly, for fixed $p$, $2<p<\infty$ from Lemma \ref{loguisweaksoln}, $\zeta=v=\log f(\nabla u)$ is a weak sub solution to (\ref{Lzeta0}), $L\zeta=0$ in $D$. 
Using this observation and $g^{2m-1}\tilde{u}\geq 0$ as a test function and the fact that $g\equiv 0$ on $B(0,2)$, we have 
\begin{align}
\label{III'+III''}
 \begin{split}
0 &\geq 2m\int\limits_{\Omega(t)}\sum\limits_{k,j=1}^{2}\mathfrak{b}_{kj}\frac{\urpartial}{\urpartial x_{j}}(\log f(\nabla u))\frac{\urpartial }{\urpartial x_{k}}\left(g^{2m-1} \tilde{u}\right)\rd\nu \\
  &=-2m\int\limits_{\Omega(t)}\sum\limits_{k,j=1}^{2}\mathfrak{b}_{kj}g_{x_{j}}\frac{\urpartial }{\urpartial x_{k}}\left(g^{2m-1}(u-t)\right)\rd\nu \\
&=-2m(2m-1)\int\limits_{\Omega(t)}\sum\limits_{k,j=1}^{2}\mathfrak{b}_{kj}g_{x_{j}}g_{x_{k}}g^{2m-2}(u-t)\rd\nu \\
 &\, \, +2m\int\limits_{\Omega(t)}\sum\limits_{k,j=1}^{2}\mathfrak{b}_{kj}g_{x_{j}}g^{2m-1}u_{x_{k}}\rd\nu\\
	&=-(III'+III'').
\end{split}
\end{align}
Arguing as in the previous case we have (\ref{II''combined}) when $p>2$. 
Therefore, for fixed $p$, $1<p<\infty$, (\ref{II''combined}), Lemma \ref{fishomdp}, and Euler's formula for a homogenous function yield

\begin{align}
\label{estimateforbothI}
 \begin{split}
\int\limits_{\{z:\, u(z)=t\}}g^{2m}\frac{f(\nabla u)}{|\nabla u|}\rd H^{1}(z) &=\frac{1}{p(p-1)}\int\limits_{\{z\in D:\, \, u(z)=t\}}\sum\limits_{k,j=1}^{2}\mathfrak{b}_{kj}\frac{u_{x_{k}}u_{x_{j}}}{|\nabla u|}g^{2m}\rd H^{1}(z) \\
&\leq \frac{2m(2m-1)}{p(p-1)}\int\limits_{\Omega(t)}\sum\limits_{k,j=1}^{2}\mathfrak{b}_{kj}g_{x_{j}}g_{x_{k}}g^{2m-2}(u-t)\rd\nu \\
&\leq c\, 2m (2m-1)\int\limits_{\Omega(t)} |\nabla u|^{p-2} |\nabla g|^{2} g^{2m-2} u\, \rd\nu. 
 \end{split}
\end{align}

Let $\{Q_{i}\}$ be a closed Whitney cube decomposition of $\Omega(t)$ and let $z_{i}$ be the center of $Q_{i}$ for $i=1,\ldots$. Let $R_{i}$ be the union of 
cubes that have a common point in the boundary with $Q_{i}$. 

Note that the definition of $g$ and Lemma \ref{fishomdp} yield for a.e $z\in\Omega$
\begin{align}
\label{derofgis}
|\nabla g|\leq c \frac{|\nabla f(\nabla u)| \|D^{2}u\|}{f(\nabla u)}\approx \frac{\|D^{2}u\|}{|\nabla u|}.
\end{align}
 Moreover, it easily follows from Lemma \ref{localholderfornablau} that 
 \begin{align}
 \label{regonwhitneycubes}
 \int\limits_{Q_{i}}|\nabla u|^{p-2}\sum\limits_{k,j}\left(u_{x_{k}x_{j}}\right)^{2}\rd\nu\leq c\, \int\limits_{R_{i}}\frac{|\nabla u|^{p}}{d(z, \partial\Omega(t))}\rd\nu
 \end{align}
 for every $i=1,\ldots$.

Using (\ref{derofgis}), (\ref{regonwhitneycubes}), Lemmas \ref{fishomdp}, \ref{finequality} in (\ref{estimateforbothI}) on the Whitney cubes $Q_{i}$ we see that
\begin{align}
\label{longintegrals}
 \begin{split}
\int\limits_{\{z:\, u(z)=t\}}g^{2m}\frac{f(\nabla u)}{|\nabla u|}\rd H^{1} & \leq c' \, m^{2}\int\limits_{\Omega(t)} u |\nabla u|^{p-2} |\nabla g|^{2} g^{2m-2}\rd\nu \\
&\leq c' \, m^{2} \sum\limits_{i}\esssup\limits_{Q_{i}} \left(\frac{u}{|\nabla u|^{2}}g^{2m-2}\right)\int\limits_{Q_{i}}|\nabla u|^{p} |\nabla g|^{2}\rd\nu \\
&\leq c' \, m^{2} \sum\limits_{i}\esssup\limits_{Q_{i}} \left(\frac{u}{|\nabla u|^{2}}g^{2m-2}\right)\int\limits_{Q_{i}}|\nabla u|^{p} \frac{|D^{2}u|^{2}}{|\nabla u|^{2}}\rd\nu \\
&\leq c' \, m^{2} \sum\limits_{i}\esssup\limits_{Q_{i}} \left(\frac{u}{|\nabla u|^{2}}g^{2m-2}\right)\int\limits_{Q_{i}}|\nabla u|^{p-2} |D^{2}u|^{2}\rd\nu \\
&\leq c' \, m^{2} \sum\limits_{i}\esssup\limits_{Q_{i}} \left(\frac{u}{|\nabla u|^{2}}g^{2m-2}\right)\int\limits_{R_{i}}\frac{|\nabla u|^{p}}{(d(z,\partial\Omega))^{2}}\rd\nu \\
&\leq c' \, m^{2} \sum\limits_{i}\esssup\limits_{Q_{i}}\left(g^{2m-2}\right)\int\limits_{R_{i}} u |\nabla u|^{p-2} \frac{1}{(d(z,\partial\Omega))^{2}}\rd\nu \\
&\leq c'\,  m^{2} \sum\limits_{i}\esssup\limits_{Q_{i}}\left(g^{2m-2}\right)\int\limits_{R_{i}} u |\nabla u|^{p-2} \frac{|\nabla u|^{2}}{u^{2}}\rd\nu \\
&\leq c'\,  m^{2} \sum\limits_{i}\esssup\limits_{Q_{i}}\left(g^{2m-2}\right)\int\limits_{R_{i}} \frac{|\nabla u|^{p}}{u}\rd\nu \\
&\leq c'\,  m^{2} \int\limits_{\Omega(t)}(g+\tilde{c})^{2m-2} \frac{f(\nabla u)}{u}\rd\nu.
 \end{split}
\end{align} 
Here we have used the fact that $Q_{i}$ intersects with finitely many $R_{i}$ which allows us to interchange freely $R_{i}$ and $Q_{i}$.

Moreover, Lemmas \ref{finequality}, \ref{uisholder} yield
\begin{align}
 \log f(\nabla u)\approx \log|\nabla u| \leq \log(c\, \frac{u(z)}{d(z, \partial\Omega(z))}) \leq \log(c\, u^{\frac{1}{\alpha}-1})\leq \hat{c}\, \log(\frac{1}{t})
\end{align}
whenever $z\in \{\tilde{z}\in D:\, \, u(\tilde{z})=t\}$ and $0<t<1/2$. Therefore, for $z\in \{\tilde{z}\in D:\, \, u(\tilde{z})=t\}$ and $0<t<1/2$ 
we see from Lemmas \ref{finequality}, \ref{harnackinequality} that
\begin{align}
\label{powertwoofgc}
(g+\tilde{c})^{2m-2}=(g^{2}+2g\tilde{c}+\tilde{c}^{2})^{m-1}\leq (g^{2}+c\log 1/t)^{m-1}. 
\end{align}
whenever $0<t<1/2$.
Using the Binomial theorem and (\ref{powertwoofgc}) we can write
\begin{align}
 (g^{2}+c\log 1/t)^{m-1} = \sum\limits_{k=0}^{m-1}\tfrac{(m-1)!}{k!(m-k-1)!}g^{2k}(c\log \frac{1}{t})^{m-1-k}. \label{binom}
\end{align}

Let 
\[
 I_{m}(t)=\int\limits_{\{z:\, u(z)=t\}}g^{2m}\frac{f(\nabla u)}{|\nabla u|}dH^{1}(z)\, \, \mbox{for}\, \, 0<t<\frac{1}{2}.
\]
Then using the Coarea formula, (\ref{estimateforbothI}), (\ref{longintegrals}) and (\ref{binom}) we obtain
\begin{align}
\label{imisless}
\begin{split}
I_{m}(t)&=\int\limits_{\{z:\, u(z)=t\}}g^{2m}\frac{f(\nabla u)}{|\nabla u|}\rd H^{1}(z)\\
&\leq c'\, m^{2}\int\limits_{\Omega(t)} (g+c)^{2m-2} \frac{f(\nabla u)}{u} \rd\nu \\
&=c'\, m^{2}\int\limits_{t}^{1}\frac{1}{\tau}\left(\int\limits_{\{z:\, u(z)=\tau\}} (g+c)^{2m-2}\frac{f(\nabla u)}{|\nabla u|}\rd H^{1}(z)\right) \rd\tau \\
&\leq c'\, m^{2}\int\limits_{t}^{1}\frac{1}{\tau}\left(\int\limits_{\{z:\, u(z)=\tau\}} \sum\limits_{k=0}^{m-1}\tfrac{(m-1)!}{k!(m-k-1)!}g^{2k}(c\log \frac{1}{\tau})^{m-1-k}\frac{f(\nabla u)}{|\nabla u|}\rd H^{1}(z)\right) \rd\tau \\
&\leq  c'\, m^{2}\sum\limits_{k=0}^{m-1}\tfrac{(m-1)!}{k!(m-k-1)!}\int\limits_{t}^{1}\frac{(c\log \frac{1}{\tau})^{m-1-k}}{\tau}\left(\int\limits_{\{z:\, u(z)=\tau\}} g^{2k}\frac{f(\nabla u)}{|\nabla u|}\rd H^{1}(z)\right) \rd\tau \\
&\leq c'\, m^{2}\sum\limits_{k=0}^{m-1}\tfrac{(m-1)!}{k!(m-k-1)!}\left[\int\limits_{t}^{1}\frac{(c\log \frac{1}{\tau})^{m-1-k}}{\tau} I_{k} \rd\tau\right].
\end{split}
\end{align}
It easily follows from $\nabla\cdot \nabla f(\nabla u(z))=0$ for a.e $z\in D$, homogeneity of $f$ and the divergence theorem that
\begin{align}
\label{I0isless}
\begin{split}
 I_{0}(t)=\int\limits_{\{z:\, u(z)=t\}}\frac{f(\nabla u)}{|\nabla u|}\rd H^{1}(z)= \mathrm{constant}=c(p,f) \, \, \mbox{for}\, \, 0<t<1.
\end{split}
\end{align}

One can now use an induction argument on $m$ in the following way: by (\ref{I0isless}) we have $I_{0}\leq c_{\ast}$ for  $0<t<1/2$, and next assume that we have 
\begin{align}
 I_{k}\leq c_{\ast}^{k+1}k![\log\frac{1}{t}]^{k}\, \, \mbox{when}\, \, 0<t<\frac{1}{2} \, \, \mbox{and for every}\, \,  1\leq k \leq m-1,
\end{align}
where  $1\leq c_{\ast}$. Then for $k=m$ a positive integer we have
\begin{align}
\label{inductiononm}
\begin{split}
I_{m}(t) & \leq c'\, m^{2}\sum\limits_{k=0}^{m-1}\frac{(m-1)!}{k!(m-k-1)!}\left[\int\limits_{t}^{1}\frac{(c\log \frac{1}{\tau})^{m-1-k}}{\tau} I_{k} \rd\tau\right] \\
	& \leq c'\, m^{2}\sum\limits_{k=0}^{m-1}\frac{(m-1)!}{k!(m-k-1)!}\left[\int\limits_{t}^{1}\frac{(c\log \frac{1}{\tau})^{m-1-k}}{\tau} c_{\ast}^{k+1}k!(\log(\frac{1}{\tau}))^{k} \rd\tau\right] \\
	& \leq c'\, m^{2}\sum\limits_{k=0}^{m-1}\frac{(m-1)!}{k!(m-k-1)!}c^{m-k-1} c_{\ast}^{k+1}k!\left[\int\limits_{t}^{1}\frac{(\log \frac{1}{\tau})^{m-1}}{\tau} \rd\tau\right] \\
	&\leq c'\, m^{2}\sum\limits_{k=0}^{m-1}\frac{(m-1)!}{k!(m-k-1)!}c^{m-k-1} c_{\ast}^{k+1}k!\frac{(\log(\frac{1}{t}))^{m}}{m} \\
	& \leq c'\, c_{\ast}^{m}m!(\log \frac{1}{t})^{m}\left(\sum\limits_{k=0}^{m-1}\frac{1}{(m-k-1)!}\right)\\
	&\leq c_{\ast}^{m+1}m!(\log \frac{1}{t})^{m}.
\end{split}
\end{align}
for $0<t<1/2$, and $c_{\ast}$ large enough.

Hence by (\ref{inductiononm}), Lemma \ref{mainlemma} is true with $w$ replaced by $g$. It follows from $w\leq g+c'$ that Lemma \ref{mainlemma} is also true for $w$.
\end{proof}
By Lemma \ref{mainlemma} we get for $0<t<1/2$
\begin{align}
\label{divideeverything}
\begin{split}
 \int\limits_{\{z\in D: \,\, u(z)=t\}} \frac{f(\nabla u)}{|\nabla u|}\frac{w^{2m}}{(2c _{\ast})^{m}m![\log\frac{1}{t}]^{m}}\rd H^{1}(z) \leq 2^{-m}c_{\ast}.
\end{split}
\end{align}
Summing over $m$ in (\ref{divideeverything}) yields for $0<t<1/2$
\begin{align}
\label{summingoverm}
\begin{split}
 \int\limits_{\{z\in D:\,\, u(z)=t\}} \frac{f(\nabla u)}{|\nabla u|}\, \mathrm{exp}\left[\frac{w^{2}}{2c _{\ast} \log\frac{1}{t}}\right]\rd H^{1}(z) \leq 2c _{\ast} .
\end{split}
\end{align}
Define
\begin{align}
\mathfrak{D}(t)=\sqrt{4c _{\ast} \left(\log\frac{1}{t}\right)\left(\log\log\frac{1}{t}\right)}\, \ \mathrm{for}\, \ 0<t<e^{-2},
\end{align}
and
\begin{align}
  \mathfrak{B}(t)=\{z:\, u(z)=t\, \, \mathrm{and}\, \ w(z)\geq \mathfrak{D}(t)\}.
\end{align}
Then by (\ref{summingoverm}) we have
\begin{align}
\label{2cstarstar}
\begin{split}
2c_{\ast} & \geq \int\limits_{\{z\in D:\, \, u(z)=t\}} \frac{f(\nabla u)}{|\nabla u|}\, \mathrm{exp}\left[\frac{w^{2}}{2c _{\ast} \log\frac{1}{t}}\right]\rd H^{1}(z) \\
& \geq \int\limits_{\mathfrak{B}(t)} \frac{f(\nabla u)}{|\nabla u|}\, \mathrm{exp}\left[\frac{w^{2}}{2c _{\ast} \log\frac{1}{t}}\right]\rd H^{1}(z)  \\
& \geq \int\limits_{\mathfrak{B}(t)} \frac{f(\nabla u)}{|\nabla u|}\, \mathrm{exp}\left[\frac{\alpha^{2}}{2c _{\ast} \log\frac{1}{t}}\right]\rd H^{1}(z)  \\
& = \int\limits_{\mathfrak{B}(t)} \frac{f(\nabla u)}{|\nabla u|}\, (-\log t)^{2} \rd H^{1}(z).
\end{split}
\end{align}
We conclude from (\ref{2cstarstar}) that 
\begin{align}
\label{441}
\begin{split}
 \int\limits_{\mathfrak{B}(t)}\frac{f(\nabla u)}{|\nabla u|}\rd H^{1}(z)\leq \frac{2c_{**}}{\left(\log\frac{1}{t}\right)^{2}}. 
\end{split}
\end{align}
For a fixed and large $A$, we define the Hausdorff measure $H^{\lambda}$ as follows;

Let
\begin{align}
\label{defnoflambda}
   \lambda(r) = \left\{
     \begin{array}{ll}
       r\re^{A\mathfrak{D}(r)} &\mathrm{when}\, \,  1< p\leq 2 \\
       r\re^{-A\mathfrak{D}(r)} &\mathrm{when}\, \,  2\leq p <\infty.
     \end{array}
   \right.
\end{align} 
Let Hausdorff $H^{\lambda}$ measure and Hausdorff dimension of 
a measure be as defined before Theorem \ref{carleson} relative $\lambda$ as in (\ref{defnoflambda}).

We can now follow closely the argument in \cite[Section 3]{LNP} and deduce that Theorem \ref{maintheorem} is true. For the reader's convenience we give the argument.
\begin{proof}[Proof of Theorem \ref{maintheorem}]
To prove Theorem \ref{maintheorem} for fixed $p$, $1<p\leq 2$ we show that
for a large $A$, 
 $\mu$ is absolutely continuous with respect to $H^{\lambda}$ measure. To this end, let $E\subset\partial\Omega$ be a Borel 
set with $H^{\lambda}(E)=0$. Let $E=E_{1}\cup E_{2}$ where
\begin{align}
E_{1}:=\{z\in E;\, \limsup\limits_{r\to 0}\frac{\mu(B(z,r))}{\lambda(r)}<\infty\},
\end{align}
and
\begin{align}
\label{e2}
\begin{split}
 E_{2}:=\{z\in E;\, \limsup\limits_{r\to 0}\frac{\mu(B(z,r))}{\lambda(r)}=\infty\}.
\end{split}
\end{align} 
It is easily shown that $\mu(E_{1})=0$. It remains to show that
$\mu(E_{2})=0$. By measure theoretic arguments, definition of $\lambda$  and by Vitali's covering argument
it can be shown that given $0<r_{0}<10^{-100}$ there is $\{r_{i}<r_{0}/100, z_{i}\in E_{2}\}$ 
such that 
\begin{align}
\label{observationbyvitali}
 \begin{split}
&B(z_{i}, 10r_{i})\, \, \,  \mbox{are disjoint balls},\\
&\{B(z_{i}, 100r_{i})\}\, \, \mbox{is a covering for}\, \, E_{2}, \\
&\mu(B(z_{i},100r_{i}))\leq 10^{9} \mu(B(z_{i},r_{i}))\, \ \mathrm{and}\, \ \lambda(100s)\leq \mu(B(z,s))\, \, \mbox{for every}\, \, i
 \end{split}
\end{align}
(see \cite[Proof of Theorem 1.3]{LNP}).

Choose $\zeta_{i}\in\partial B(z_{i}, 2r_{i})$ such that $u(\zeta_{i})=\max u$ on $\overline{B(z_{i}, 2r_{i})}$. 
From the last line of (\ref{observationbyvitali}) and Lemma \ref{uismu} 
we know that
the maximum of $u$ on $\overline{B(z_{i},2r_{i})}$ and the maximum of $u$ on $\overline{B(z_{i},5r_{i})}$ are proportional.
Thus, this observation and Lemma \ref{boundaryharnackinequalty} yield $d(\zeta_{i}, \partial\Omega)\approx r_{i}$.

Moreover, using $d(\zeta_{i}, \partial\Omega)\approx r_{i}$ and Lemmas \ref{fishomdp}, \ref{uismu} we see for fixed $i$ that
\begin{align}
\label{uwtildeholds}
  \frac{\mu(B(z_{i},10r_{i}))}{r_{i}}\approx \left(\frac{u(\zeta_{i})}{d(\zeta_{i},\partial\Omega)}\right)^{p-1}\approx \frac{f(\nabla u(z))}{|\nabla u(z)|}
\end{align}
whenever $z\in B(\zeta_{i}, d(\zeta_{i},\partial\Omega)/2)$. Choose $m$ so that $2^{-m}\leq u(\zeta_{i})\leq 2^{-m+1}$, and 
let $\eta_{i}$ be the first point on the line segment from 
$\zeta_{i}$ to a point on $\partial\Omega\cap \partial B(\zeta_{i}, d(\zeta_{i}, \partial\Omega))$ satisfying $u(\eta_{i})=2^{-m}$. Then we see that 
(\ref{uwtildeholds}) holds with
$\zeta_{i}$ replaced by $\eta_{i}$. That is,
\begin{align}
\label{614}
\begin{split}
& 	u(\eta_{i})=2^{-m}\, \, \mbox{and}\, \, d(\eta_{i}, \partial\Omega)\approx r_{i},\\
&   \frac{\mu(B(z_{i},10r_{i}))}{r_{i}}\approx \left(\frac{u(\eta_{i})}{d(\eta_{i},\partial\Omega)}\right)^{p-1}\approx \frac{f(\nabla u(z))}{|\nabla u(z)|} \approx |\nabla u(z)|^{p-1}
\end{split}
\end{align}
whenever $z\in B(\eta_{i}, d(\eta_{i},\partial\Omega)/2)$.

From (\ref{observationbyvitali}) and (\ref{614}) for $z\in B(\eta_{i}, d(\eta_{i},\partial\Omega)/2)$ we have
\begin{align}
\label{447}
\begin{split}
A\mathfrak{D}(100r_{i}) &= \log\left(\frac{\lambda(100r_{i})}{100 r_{i}}\right)  \\
		    & \leq \log\left(\frac{\mu(B(z_{i}, r_{i}))}{100 r_{i}}\right) \\
		    &\leq c \log\left(\frac{\mu(B(z_{i}, 10 r_{i}))}{r_{i}}\right) \\
		    &\leq c \log|\nabla u|^{p-1} \leq c \log f(\nabla u)+c=w(z)+c.
\end{split}
\end{align}
where $A$ is as in (\ref{defnoflambda}) and $c=c(p,f)\geq 1$.

Using Lemma \ref{uisholder} we can estimate $2^{-m}$ above in terms of $r_{i}$. We can also estimate $2^{-m}$ 
below in terms of $r_{i}$ using the last line in (\ref{observationbyvitali}) and (\ref{614}).
That is, there exist $c'=c(p,f)$ and $\beta=\beta(p,f)<1$ such that 
\begin{align}
\label{450}
r_{i}\leq c' (2^{-m})^{\beta}  \, \ \mathrm{and}\,\ 2^{-m}\leq c' r_{i}^{\beta}. 
\end{align}
From (\ref{441}), (\ref{447})-(\ref{450}) we have,
\begin{align}
\label{313}
\begin{split}
\mu[B(z_{i},  10r_{i})]\leq c \int\limits_{\mathfrak{B}(2^{-m})\cap B(z_{i}, 10r_{i})} \frac{f(\nabla u)}{|\nabla u|} \rd H^{1}(z) 
\end{split}
\end{align}
For large $A$, (\ref{441}), (\ref{450}), and (\ref{313}) yield
\begin{align}
\begin{split}
\mu(E_{2}) & \leq \mu\left(\bigcup\limits_{i} B(z_{i}, 100r_{i})\right) \\
& \leq 10^{9} \sum\limits_{i} \mu(B(z_{i}, 10r_{i}))  \\
&\leq c \sum\limits_{m=m_{0}}\int\limits_{\mathfrak{B}(2^{-m})} \frac{f(\nabla u)}{|\nabla u|} \rd H^{1}(z) \\
&\leq c^{2}\sum\limits_{m=m_{0}} m^{-2} \leq \frac{c^{3}}{m_{0}}.  
\end{split}
\end{align}
where $2^{-m_{0}\beta}=cr^{\beta^{2}}_{0}$. As $r_{0}\to 0$ we have $\mu(E_{2})\to 0$. So we have the desired result when $1<p\leq 2$.


To finish the proof of Theorem \ref{maintheorem}, it remains to show that for $2 \leq p<\infty$, $\mu$ is concentrated on a set of $\sigma-$finite $H^{\lambda}$ measure. 
To obtain this, by definition, we show that there is a Borel set $K\subset\partial\Omega$ having $\sigma-$finite $H^{\lambda}$ 
measure satisfying $\mu(K)=\mu(\partial\Omega)$.

We first show that $\mu(K')=0$ where
\begin{align}
\label{K'haszero}
\begin{split}
 K':=\{z\in \partial\Omega; \, \lim\limits_{r\to 0}\frac{\mu(B(z,r)}{\lambda(r)}=0\}.
\end{split}
\end{align}
Then $\mu(K)=\mu(\partial\Omega)$ where  
\[
K=\{z\in\partial\Omega;\, \limsup\limits_{r\to 0}\frac{\mu(B(z,r)}{\lambda(r)}>0\}
\]
and it will follow easily that $K$ has $\sigma-$finite $H^{\lambda}$ measure.

Let $r_{0}$ be sufficiently small. We can argue as in \cite[Proof of Lemma 2.4]{LNP} to find $\{r_{i}<r_{0}/100, z_{i}\in K'\}$ such that 
\begin{align}
\label{observationbyvitaliforp2}
 \begin{split}
&B(z_{i}, 10r_{i})\, \, \,  \mbox{are disjoint balls},\\
&\{B(z_{i}, 100r_{i})\}\, \, \mbox{is a covering for}\, \, K', \\
&\mu(B(z_{i},100r_{i}))\leq c \mu(B(z_{i},r_{i})) \, \ \mathrm{and}\, \ \ \mu(B(z_{i},100r_{i}))\leq \lambda(r_{i})\, \,\mbox{for every}\, \, i.
 \end{split}
\end{align}
where the constant is independent of $z_{i}$ and $r_{i}$ for $i=1,\ldots$.
Let $I'$ be the set of all indexes $i$ for which $r_{i}^{3}\leq \mu(B(z_{i}, 100r_{i}))$ and let $I''$ be the indexes where this inequality does not hold. 
By (\ref{observationbyvitaliforp2}) we see that
\begin{align}
\label{K'isleqthatn}
\begin{split}
\mu(K') &\leq \mu (\bigcup\limits_{i\in I'}B(z_{i}, 100r_{i})) \\
	& + \mu(\bigcup\limits_{i\in I'}B(z_{i}, 100r_{i})) + \mu(\bigcup\limits_{i\in I''}B(z_{i}, 100r_{i})) \\
	& \leq \mu(\bigcup\limits_{i\in I'}B(z_{i}, 100r_{i}))+\sum\limits_{i\in I''}r^{3}_{i} \\
	& \leq \mu(\bigcup\limits_{i\in I'}B(z_{i}, 100r_{i}))+c' r_{0}H^{2}(\Omega).
\end{split}
\end{align}
When $i\in I'$ we can repeat the argument for $1<p\leq 2$ to get (\ref{313}). Finally, using (\ref{441}) and (\ref{313}) in (\ref{K'isleqthatn})  we see that

\begin{align}
\begin{split}
\mu(K')-c'r_{0}H^{2}(\Omega) & \leq \mu(\bigcup\limits_{i\in I'}B(z_{i}, 100r_{i})) \\
			      & \leq c \sum\limits_{i \in I'}\mu(B(z_{i}, 10r_{i})) \\
			      & \leq c \sum\limits_{m=m_{0}}\int\limits_{\mathfrak{B}(2^{-m})}\frac{f(\nabla u)}{|\nabla u|}\rd H^{1} \\
			      &\leq c^{2}\sum\limits_{m=m_{0}} m^{-2} \leq \frac{c^{3}}{m_{0}}. 
\end{split}
\end{align}
Hence $2^{-m_{0}\beta}=cr^{\beta^{2}}_{0}$. Since $r_{0}$ can be arbitrarily small, 
we can let $r_{0}\to 0$ from which we conclude that $\mu(K')=0$.

It remains to show that $\mu(K)=\mu(\partial\Omega)$ and $K$ has $\sigma-$finite $H^{\lambda}$ measure. To this end let $K_{i}$, 
for a positive integer $i$, be the set of points in $K$ with the property that
\[
K_{i}=\{z\in\partial\Omega;\, \limsup\limits_{r\to 0}\frac{\mu(B(z,r)}{\lambda(r)}\geq \frac{1}{i}\}.
\]
From a covering argument it follows that
\[
 H^{\lambda}(K_{i}) \leq c\, i\mu(K_{i})
\]
from which we can conclude that $K_{i}$ has $\sigma-$finite $H^{\lambda}$ measure. 
Since $\bigcup\limits_{i} K_{i}=K$, we conclude that $K$ has $\sigma-$finite $H^{\lambda}$ measure.
which finishes the proof for $2\leq p<\infty$.

The proof of Theorem \ref{maintheorem} is now complete.
\end{proof}
  \section*{Acknowledgments} The author would like to thank John L. Lewis for originally suggesting the problem under consideration and for many helpful discussions. 
  The author would also like to thank Leonid Kovalev for some helpful comments regarding regularity assumptions on $f$.


\bibliographystyle{amsplain}
\bibliography{myref}
\end{document}